\documentclass[11pt]{amsart}
\def\a{\alpha}
\usepackage[all]{xy}
\usepackage{amsmath}
\usepackage[latin1]{inputenc}
\usepackage{amsfonts}
\usepackage{amssymb}
\usepackage{multicol}
\usepackage{verbatim}
\usepackage{amsthm}
\usepackage{amscd}
\usepackage{pb-diagram}
\usepackage[mathscr]{eucal}
\usepackage[latin1]{inputenc}

\usepackage{amsmath}

\usepackage{amsfonts}

\usepackage{amssymb}

\usepackage{amsthm}

\usepackage{graphics}

\newcommand{\ZZ}{\mathbb{Z}}

\newcommand{\CC}{\mathbb{C}}

\newcommand{\QQ}{\mathbb{Q}}

\newcommand{\Glie}{\mathfrak{g}}

\newcommand{\Hlie}{\mathfrak{h}}

\newcommand{\U}{\mathcal{U}}

\newtheorem{thm}{Theorem}[section]

\newtheorem{defi}[thm]{Definition}

\newtheorem{cor}[thm]{Corollary}

\newtheorem{prop}[thm]{Proposition}

\newtheorem{rem}[thm]{Remark}

\newtheorem{ex}[thm]{Example}

\newcommand{\C}{\mathbb{C}}
\newcommand{\Q}{\mathbb{Q}}
\newcommand{\Z}{\mathbb{Z}}

\newcommand{\g}{\mathfrak{g}}
\newcommand{\bo}{\mathfrak{b}}
\newcommand{\tb}{\mathbf{\mathfrak{t}}}
\newcommand{\ga}{\overline{\alpha}}

\newcommand{\Psib}{\mbox{\boldmath$\Psi$}}
\newcommand{\Psibs}{\scalebox{.7}{\boldmath$\Psi$}}
\newcommand{\qbin}[2]{{\left[
\begin{matrix}{\,\displaystyle #1\,}\\
{\,\displaystyle #2\,}\end{matrix}
\right]
}}

\newcommand{\nc}{\newcommand}
\nc{\on}{\operatorname}
\nc{\la}{\lambda}
\nc{\wh}{\widehat}
\nc{\wt}{\widetilde}
\nc{\sw}{{\mathfrak s}{\mathfrak l}}
\nc{\ghat}{\wh{\g}}
\nc{\hhat}{\wh{\h}}
\nc{\mc}{\mathcal}
\nc{\bi}{\bibitem}
\nc{\pa}{\partial}
\nc{\ppart}{(\!(t)\!)}
\nc{\pparl}{(\!(\la)\!)}
\nc{\zpart}{(\!(z^{-1})\!)}
\nc{\n}{{\mathfrak n}}
\nc{\ol}{\overline}
\nc{\mb}{\mathbf}
\nc{\bb}{{\mathfrak b}}
\nc{\su}{\wh\sw_2}
\nc{\h}{{\mathfrak h}}
\nc{\can}{\on{can}}
\nc{\ntil}{\wt{\n}}
\nc{\pone}{{\mathbb P}^1}
\nc{\bs}{\backslash}
\nc{\al}{\alpha}
\nc{\gt}{{\mathfrak g}'}
\nc{\ds}{\displaystyle}

\usepackage[all]{xy}
\title{
Stable maps, Q-operators and category $\mathcal{O}$}
\author{David Hernandez}

\address{Universit\'e de Paris and Sorbonne Universit\'e, CNRS, IMJ-PRG, IUF, F-75006 Paris, France.}

\email{david.hernandez@u-paris.fr}

\begin{document} 

\begin{abstract} Motivated by Maulik-Okounkov stable maps associated to quiver varieties, 
we define and construct algebraic stable maps on tensor products of representations in the category $\mathcal{O}$ of 
the Borel subalgebra of an arbitrary untwisted quantum affine algebra. Our representation-theoretical construction is 
based on the study of the action of Cartan-Drinfeld subalgebras. We prove the algebraic stable maps are invertible and 
depend rationally on the spectral parameter. As an application, we obtain new $R$-matrices in the category $\mathcal{O}$ 
and we establish that a large family of 
simple modules, including the prefundamental representations associated to $Q$-operators, generically commute as representations of 
the Cartan-Drinfeld subalgebra. We also establish categorified $QQ^*$-systems in terms of the $R$-matrices we construct.
\end{abstract}

\maketitle

\tableofcontents

\section{Introduction}

Let $q\in\CC^*$ which is not a root of unity and let $\U_q(\Glie)$ be an untwisted quantum affine algebra. The category $\mathcal{C}$ of finite-dimensional
representations of $\U_q(\Glie)$ has been studied from various geometric, algebraic and combinatorial points of view. One crucial property
of the category $\mathcal{C}$, which goes back to Drinfeld, is to admit generic braidings, that is there is an isomorphism 
$$V\otimes W \simeq W \otimes V$$ 
for generic simple modules in $\mathcal{C}$. Such isomorphisms are called $R$-matrices and satisfy the Yang-Baxter equation. Moreover the tensor product $V\otimes W$ is generically simple. These results follows from the existence of the universal $R$-matrix of $\U_q(\Glie)$. 

Maulik and Okounkov \cite{mo} proposed a striking new point of view on these structures by introducing the notion of stable envelopes and stable maps.
These authors have presented a very general construction such maps 
$$\text{Stab}_{\mathfrak{C}} : K_T(X^A)\rightarrow K_T(X)$$
based\footnote{For the moment, only the
cohomological version of the work of Maulik-Okounkov is public yet. For $K$-theoretic stable map there are several important differences with cohomological versions, in particular they depend on a new parameter, the slope, see \cite{os}.} on remarkable Lagrangian correspondences in $X\times X^A$ defined from the action of a pair of tori $A\subset T$ on a symplectic variety $X$ (the action of
$A$ is supposed to preserve the symplectic form). Here $K_T$ denotes the equivariant $K$-theory with respect to $T$ and $X^A$ the fixed point locus for the $A$-action. The Lagrangian sub-varieties, the stable envelopes, are built by successive approximations from the closure of a natural preimage of a diagonal subvariety. This holds in great generality including symplectic resolutions. 

The construction of stable maps depends on some additional data, in particular on a cone $\mathfrak{C}\subset \text{Lie}(A)$, the chamber, which is a connected component in the Lie algebra of $A$ 
of the complementary of an hyperplan arrangement. The choice of $\mathfrak{C}$ leads to the definition of attracting directions in the normal direction to $X^A$ and determines the support of 
the stable envelope. The stable map $\text{Stab}_\mathfrak{C}$ satisfies a certain triangularity property with respect to $\prec_{\mathfrak{C}}$ . 
This "topological" triangularity is a crucial property of stable maps.


For a choice of two chambers $\mathfrak{C}$ and $\mathfrak{C}'$, the construction gives two maps :
$$\xymatrix{ & K_T(X) & \\ K_T(X^A)\ar[ur]_{\text{Stab}_{\mathfrak{C}}}\ar@{-->}[rr]_{\mathcal{R}_{\mathfrak{C}',\mathfrak{C}}}& & K_T(X^A)\ar[ul]^{\text{Stab}_{\mathfrak{C}'}}}.$$
Up to localization, the map $\text{Stab}_{\mathfrak{C'}}$ is invertible and we get a geometric $R$-matrix
$$R_{\mathfrak{C}',\mathfrak{C}} = (\text{Stab}_{\mathfrak{C}'})^{-1}\circ \text{Stab}_{\mathfrak{C}} \in \text{End}(K_T(X^A))$$
which might be seen as a wall-crossing from the chamber $\mathfrak{C}$ to $\mathfrak{C}'$. It gives rise in particular to $R$-matrices which are already known, but the techniques which are used go much further.

The theory of stable envelopes plays an important role in geometric representation theory as well as in enumerative 
geometry and has various incarnations in various areas of mathematics. 
Nakajima varieties are particularly important examples.
Indeed in a series of seminal papers Nakajima has constructed, in the equivariant $K$-theory of these varieties, certain representations of quantum affine algebras
 $\mathcal{U}_q(\mathfrak{g})$ for $\mathfrak{g}$ simply-laced (see \cite{Nsem, Nak0}). 
Moreover, the geometric study of the coproduct \cite{VV, Ntens} leads to the construction of tensor products of 
certain finite-dimensional representations. Stable envelopes give a geometric construction of $R$-matrices for tensor 
products of fundamental representations in the category $\mathcal{C}$ of finite-dimensional representations of $\U_q(\Glie)$ \cite{mo, os}.

This leads to the question of extending the construction of stable maps to non-simply laced
quantum affine algebras as well as to representations which are 
not necessarily finite-dimensional, for instance in the category $\mathcal{O}$. 
However no geometric model is known at the moment for these situations. More generally, we may 
ask for a purely representation-theoretical or algebraic characterization of stable maps. 

Let us recall that Jimbo and the first author introduced \cite{HJ} the category $\mathcal{O}$ of representations of a Borel
subalgebra $\U_q(\bo)$ of $\U_q(\Glie)$. Finite-dimensional
representations of $\U_q(\Glie)$ are objects in this category as well as the infinite-dimensional 
prefundamental representations of $\U_q(\bo)$ constructed\footnote{Such prefundamental representations were first constructed explicitly for 
${\mathfrak g} = \hat{sl}_2$ by Bazhanov-Lukyanov-Zamolodchikov, for $\hat{sl}_3$ by Bazhanov-Hibberd-Khoroshkin and 
for $\hat{sl}_n$ with $i = 1$ by Kojima.} in \cite{HJ}. 
They are obtained as asymptotic limits of Kirillov-Reshetikhin
 modules, which form a family of simple finite-dimensional representations of $\U_q(\mathfrak{g})$. 
These prefundamental representations, denoted by $L^+_{i,a}$ and $L^-_{i,a}$, are simple $\U_q(\bo)$-modules parametrized by a complex number $a\in\CC^*$ 
and $1\leq i\leq n$, where $n$ is the rank of the underlying finite-dimensional simple Lie algebra.
The category $\mathcal{O}$ and the prefundamental representations were used by Frenkel and the first author \cite{FH} to prove
a conjecture of Frenkel-Reshetikhin \cite{Fre} on the spectra of quantum integrable systems, generalizing the existence
of Baxter's polynomials to describe these spectra beyond the case of the $XXZ$-model. The prefundamental representations
play a crucial role for theses works as the corresponding transfer-matrices are the Baxter's $Q$-operators.

Our present paper has a second main motivation : the study of tensor products of $\ell$-weight vectors of representations of quantum affine algebras.
The $\ell$-weight vectors are pseudo eigenvectors for the action of the Cartan-Drinfeld subalgebra $\U_q(\Hlie)^+\subset \U_q(\Glie)$. The study of this action is strongly related to Frenkel-Reshetikhin $q$-character theory \cite{Fre}. Note that the action of the Cartan-Drinfeld subalgebra $\U_q(\Hlie)^+$ can naturally be deformed to the action of the Baxter algebra
(see \cite[Proposition 5.5]{FH} for instance). It is well known that elements of the Cartan-Drinfeld subalgebra $\U_q(\Hlie)^+$ do not behave well with respect to
the coproduct, that is why the study of tensor product of $\ell$-weight vectors is technically involved. A tensor product of $\ell$-weight vectors is not necessarily 
an $\ell$-weight vector, and this is a source of many technical developments. This can be observed for example in the tensor product of two $2$-dimensional representations of $\U_q(\hat{sl}_2)$, see Example \ref{exh4}.

 However, thanks to a remarkable properties of the coproduct on Cartan-Drinfeld elements (see \cite{da} and Theorem \ref{apco} below), certain $\ell$-weight vectors 
in the tensor product can be decomposed into sums of pure tensor of $\ell$-weight vectors \cite{h4} (see Theorem \ref{prodlweight} below) for which a 
triangularity condition appears. 
This algebraic triangularity might be seen as an analog of the topological triangularity discussed above for stable maps.

\medskip

In the present paper we propose to define algebraic stable map directly from $\ell$-weight vectors. This representation-theoretical point of view allows to give a definition for the non simply-laced types 
as well as for the category $\mathcal{O}$. It also give a practical way to handle the algebraic stable maps (we compute several examples).

The idea is the following : let $V$, $W$ in the category $\mathcal{O}$. For $v\in V$, $w\in W$ $\ell$-weight vectors, we prove that $v\otimes w$ can be canonically perturbed to produce an $\ell$-weight vector 
$$S_{V,W}(v\otimes w)\in V\otimes W.$$ 
The different terms added to $v\otimes w$ in order to obtain the new $\ell$-weight vector $S_{V,W}(v\otimes w)$ in $V\otimes W$ 
might be seen as algebraic analogs of the successive approximations in the construction of the stable envelopes mentioned above.
Moreover, a key point is that our construction respects a triangularity property for a certain partial ordering on the cartesian square of the 
integral weight lattice (Equation \ref{crosordering}), by analogy to the topological triangularity discussed above. 

We establish this defines a linear morphism
$$S_{V,W} : V\otimes W\rightarrow V\otimes W.$$
In certain cases for which the Maulik-Okounkov stable maps can be computed \cite{os}, it can be checked that they coincide with $S_{V,W}$ 
(up to a renormalization by a diagonal operator). This is expected to be true in general.

It is well known that a representation $V$ of $\U_q(\Glie)$ can be deformed by adding a spectral parameter $u$. We get a representation $V(u)$
and the algebraic stable maps deform accordingly
$$S_{V,W}(u) : V(u)\otimes W\rightarrow V(u)\otimes W.$$
We establish the algebraic stable maps are invertible and depend rationally on the spectral parameter $u$.

We have reminded above that the category $\mathcal{C}$ of finite-dimensional representations of $\U_q(\Glie)$ has generic braidings as the 
universal R-matrix can be specialized to give a meromorphic braiding, the $R$-matrix
$$\mathcal{R}_{V,W}(u) : V(u)\otimes W \rightarrow W\otimes V(u),$$
for $V$, $W$ simple finite-dimensional modules. For a generic complex number $u$ (which does not belong to a finite set),  we get an isomorphism.

But for the category $\mathcal{O}$, not only the universal $R$-matrix can not be specialized on a general tensor product of simple representations (as only one Borel subalgebra act on these representations in general), but also there are simple representations $V$, $W$ so that $V(a)\otimes W$ is non simple for any $a\in\CC^*$. Although its Grothendieck ring is commutative \cite{HJ}, the category $\mathcal{O}$ is not generically braided (see Example \ref{exnon} : in the $sl_2$-case, for any $a,b \in\CC^*$, $L_{1,a}^+\otimes L_{1,b}^-$ is not isomorphic to $L_{1,b}^-\otimes L_{1,a}^+$).
Hence, the $R$-matrices $\mathcal{R}_{V,W}(u)$ do not exist for arbitrary simple representations in the category $\mathcal{O}$. 

But the algebraic stable maps $S_{V,W}(u)$ do exist and the construction in the present paper produces maps of the form
$$\mathcal{R}_{V,W}^\alpha(u) = S_{W,V}(u) (\tau \circ \alpha(u)) (S_{V,W}(u))^{-1} : V(u)\otimes W \rightarrow W\otimes V(u)$$
where $\tau$ is the twist (and $\alpha(u)$ is a certain renormalization operator).


As an application of the results and constructions in this paper, we establish that generic tensor products of a large family of simple representations $V$, $W$ 
in the category $\mathcal{O}$ commute as representations of the Cartan-Drinfeld subalgebra $\U_q(\Hlie)^+$ : 
$$\mathcal{R}_{V,W}^1(u) : V(u)\otimes W \simeq_{\U_q(\Hlie)^+}  W\otimes V(u).$$
As far the author knows, this is a new representation-theoretical result, even in the case of prefundamental
representations (this is not a direct consequence of the commutativity of Grothendieck ring).

We also obtain that the Cartan-Drinfeld factor of the universal $R$-matrix acts rationally on a tensor product of finite-dimensional representation, up to a scalar 
multiple (this is a well-known result for the whole universal $R$-matrix).


The category $\mathcal{O}$ has a remarkable  monoidal subcategory $\mathcal{O}^-$ generated by finite-dimensional representations and negative 
prefundamental representations constructed in \cite{HL} (a dual category $\mathcal{O}^+$ is also constructed in \cite{HL}; see also \cite{FJM}). It is known \cite{FH} that prefundamental representations in the category $\mathcal{O}^-$ commute : 
$$L_{i,a}^-(u)\otimes L_{j,b}^-  \simeq L_{j,b}^- \otimes L_{i,a}^-(u)$$ 
as this tensor product is simple. 

As a consequence of the result of this paper we prove the category $\mathcal{O}^-$ admits generic braidings : for $V$, $W$ simple modules in $\mathcal{O}^-$, 
there is $\alpha$ such that $\mathcal{R}_{V,W}^\alpha(u)$ is an isomorphism of representations. By specialization, it gives non-zero morphisms 
$$\mathcal{R}_{V,W} : V\otimes W \rightarrow W\otimes V$$
which are not invertible in general. This leads to categorification 
of remarkable relations which hold in the Grothendieck ring of the category, such as the $QQ^*$-systems 
(which appear as cluster mutations and are closely related to Bethe Ansatz equations).

\medskip

Note that our results give partial informations on possible varieties for a geometric realization of prefundamental representations
\footnote{In type $A$, relations between $Q$-operators and quantum $K$-theory is discussed in \cite{psz} in the context of the theory of stable envelopes.}. We hope
it will give some additional practical tools to handle the corresponding geometric structures. Other possible further developments of the results of our paper are discussed in the last section, in particular on the polynomiality of Cartan-Drinfeld elements, 
generalized Schur-Weyl dualities in the sense of Kang-Kashiwara-Kim and natural bases of standard modules.

Note also that the category $\mathcal{O}$ studied in the present paper has been recently related \cite{hshi} to representations of
shifted quantum affine algebras in the sense of Finkelberg-Tsymbaliuk \cite{FT} and so to quantized $K$-theoretic Coulomb branches.
Hence the method developed in the present paper may also be developed in these new contexts.

\medskip

In this paper we establish various properties of the algebraic stable maps we consider. These properties are at the origin of the present work and discussions with A. Okounkov were crucial for its development (see in particular Remark \ref{remao}).

\medskip

This paper is organized as follows. In Section \ref{back} we give reminders on quantum affine algebras, their finite-dimensional
representations and the category $\mathcal{O}$ for its Borel subalgebra. In Section \ref{sm} we explain the definition and the
construction of algebraic stable maps on tensor products of modules in the category $\mathcal{O}$ (Definition \ref{defis}). 
We prove they define linear isomorphisms (Proposition \ref{isom}) and we establish the rationality in the spectral parameter (Theorem \ref{ratios}). We give explicit examples for finite and infinite dimensional representations (subsection \ref{exsldeux}). In Section \ref{rel}, we establish the compatibility of algebraic stable map with the Drinfeld coproduct for the action of Cartan-Drinfeld subalgebras (Proposition \ref{isom}). In the case of finite-dimensional representation, the algebraic stable maps are related to factors of the universal $R$-matrix (Proposition \ref{relkt}) and 
in certain remarkable cases to Maulik-Okounkov stable maps. In Section \ref{rm} the applications to the construction of $R$-matrices in the category $\mathcal{O}$ (Theorem \ref{isopre}) and categorifications of remarkable relations (Theorem \ref{catqq}) are established. In Section \ref{fd} we discuss various possible further developments.

\medskip

{\bf Acknowledgment : } The author is very grateful to Andrei Okounkov for discussions from which the idea of this paper emerged.
The author is supported by the European Research Council under the European Union's Framework Programme H2020 with ERC Grant Agreement number 647353 Qaffine.

\section{Background on quantum affine algebras}\label{back}

In this section we collect some definitions and results on quantum
affine algebras and their representations. We refer the reader to
\cite{cp} for a canonical introduction. We also discuss representations of the
Borel subalgebra of a quantum affine algebra, see \cite{HJ, FH} for
more details. In particular we remind the corresponding category $\mathcal{O}$ and the
category of finite-dimensional representations. They have been studied from many points geometric, 
algebraic, combinatorial point of views in connections to various fields, see \cite{mo, kkko, gtl, k, o} for recent developments and \cite{Hbou} for a recent review.

All vector spaces, algebras and tensor products are defined over $\CC$, except when otherwise specified.

\subsection{Quantum affine algebras and Borel algebras}\label{debut}

Let $C=(C_{i,j})_{0\leq i,j\leq n}$ be an indecomposable Cartan matrix
of untwisted affine type.  We denote by $\Glie$ the Kac--Moody
Lie algebra associated with $C$.  Set $I=\{1,\ldots, n\}$, and denote
by $\overline{\Glie}$ the finite-dimensional simple Lie algebra associated with
the Cartan matrix $(C_{i,j})_{i,j\in I}$.  Let 
$$\{\alpha_i\}_{i\in I}\text{ , }\{\alpha_i^\vee\}_{i\in I}\text{ , }\{\omega_i\}_{i\in I}\text{ , }\{\omega_i^\vee\}_{i\in I},$$ 
and $\overline{\mathfrak{h}}$ be the simple
roots, the simple coroots, the fundamental weights, the fundamental
coweights, and the Cartan subalgebra of $\overline{\Glie}$, respectively.  We will use
$$Q=\oplus_{i\in I}\Z\alpha_i\text{ , }Q^+=\oplus_{i\in I}\Z_{\ge0}\alpha_i\text{ , }P=\oplus_{i\in I}\Z\omega_i.$$  
Let $D=\mathrm{diag}(d_0\ldots, d_n)$
be the unique diagonal matrix such that $B=DC$ is symmetric and the $d_i$'s are relatively prime 
positive integers.  We will also use
$P_\Q = P\otimes \Q$ with its partial ordering defined by 
$$\omega\leq \omega'\text{ if and only if }\omega'-\omega\in Q^+.$$  
We use the numbering of the Dynkin
diagram as in \cite{kac}.  Let $a_0,\cdots,a_n$ stand for the labels as in \cite[pp.55-56]{kac}. 
We have $a_0 = 1$ and we set 
$$\alpha_0 =
-(a_1\alpha_1 + a_2\alpha_2 + \cdots + a_n\alpha_n).$$  
We fix a non-zero complex number $q$ which is not a root of unity and we set $q_i=q^{d_i}$. 
We also set $h\in\CC$ such that $q = e^h$, so that $q^r$ is well-defined for any $r\in\QQ$. 

We
will use the standard symbols for $z$ an indeterminate or a non-zero complex number which is not a root of unity :
\begin{align*}
[m]_z=\frac{z^m-z^{-m}}{z-z^{-1}}, \quad
[m]_z!=\prod_{j=1}^m[j]_z,
 \quad 
\qbin{s}{r}_z
=\frac{[s]_z!}{[r]_z![s-r]_z!}. 
\end{align*}

The quantum loop algebra $\U_q(\Glie)$ is the $\C$-algebra defined
by generators
$e_i,\ f_i,\ k_i^{\pm1}$ ($0\le i\le n$) 
and the following relations for $0\le i,j\le n$.
\begin{align*}
&k_ik_j=k_jk_i,\quad k_0^{a_0}k_1^{a_1}\cdots k_n^{a_n}=1,\quad
&k_ie_jk_i^{-1}=q_i^{C_{i,j}}e_j,\quad k_if_jk_i^{-1}=q_i^{-C_{i,j}}f_j,
\\
&[e_i,f_j]
=\delta_{i,j}\frac{k_i-k_i^{-1}}{q_i-q_i^{-1}},
\\
&\sum_{r=0}^{1-C_{i.j}}(-1)^re_i^{(1-C_{i,j}-r)}e_j e_i^{(r)}=0\quad (i\neq j),
&\sum_{r=0}^{1-C_{i.j}}(-1)^rf_i^{(1-C_{i,j}-r)}f_j f_i^{(r)}=0\quad (i\neq j)\,.
\end{align*}
Here we use the standard notations $x_i^{(r)}=x_i^r/[r]_{q_i}!$ ($x_i=e_i,f_i$). 
The algebra $\U_q(\Glie)$ has a Hopf algebra structure satisfying for $0\leq i\leq n$,
\begin{align*}
&\Delta(e_i)=e_i\otimes 1+k_i\otimes e_i,\quad
\Delta(f_i)=f_i\otimes k_i^{-1}+1\otimes f_i,
\quad
\Delta(k_i)=k_i\otimes k_i\,.
\end{align*}
The algebra $\U_q(\Glie)$ can also be presented in terms of the Drinfeld
generators \cite{Dri2, bec}
\begin{align*}
  x_{i,r}^{\pm}\ (i\in I, r\in\Z), \quad \phi_{i,\pm m}^\pm\ (i\in I,
  m\geq 0), \quad k_i^{\pm 1}\ (i\in I).
\end{align*}
It will be useful to consider the subalgebra $\mathcal{U}_q^\pm(\mathfrak{g})$ generated by the $x_{i,r}^\pm$ ($i\in I$, $r\in\Z$).
We will also use the generating series $(i\in I)$: 
$$\phi_i^\pm(z) = \sum_{m\geq 0}\phi_{i,\pm m}^\pm z^{\pm m} =
k_i^{\pm 1}\text{exp}\left(\pm (q_i - q_i^{-1})\sum_{m > 0} h_{i,\pm
    m} z^{\pm m} \right),$$
and we set $\phi_{i,\pm m}^\pm = 0$ for $m < 0$, $i\in I$.

These elements $\phi_{i,\pm m}^\pm$ are called Cartan-Drinfeld generators. They generate a subalgebra $\U_q(\Hlie)$ of $\U_q(\Glie)$.
Let $\U_q(\Hlie)^\pm$ be the subalgebra of $\U_q(\Hlie)$ generated by the $k_i$, $k_i^{-1}$ and the $h_{i,\pm r}$ ($i\in I, r > 0$).

\begin{defi} The subalgebras $\U_q(\Hlie)$ and $\U_q(\Hlie)^\pm$ are called Cartan-Drinfeld subalgebras.
\end{defi}

These algebras are commutative and will play a crucial role in this paper.

The algebra $\U_q(\Glie)$ has a $\ZZ$-grading defined by
$\on{deg}(e_i) = \on{deg}(f_i) = \on{deg}(k_i^{\pm 1}) = 0$ for $i\in
I$ and $\on{deg}(e_0) = - \on{deg}(f_0) = 1$.  It satisfies
$\on{deg}(x_{i,m}^\pm) = \on{deg}(\phi_{i,m}^\pm) = m$ for $i\in I$,
$m\in\ZZ$. For $a\in\CC^\times$, there is a corresponding algebra automorphism
$$\tau_a : \U_q(\Glie)\rightarrow \U_q(\Glie)$$ 
so that an element $g$ of degree $m\in\ZZ$ satisfies $\tau_a(g) = a^m g$.
The twist of a representation $W$ by $\tau_a$ is denoted by $W(a)$. 

We have also an automorphism $\tau_u$ of the algebra
$$\U_{q,u}(\Glie) = \U_q(\Glie)\otimes\CC(u)$$ 
defined as $\tau_a$ with $a$ replaced by the formal variable $u$. A representation $W$ of $\U_q(\Glie)$ gives rise to a 
twisted representation $W(u)$ of $\U_{q,u}(\Glie)$ (see \cite{Hbou} for detailed references). It is a $\CC(u)$-vector space 
$W(u) = W\otimes \CC(u)$. In the following, when twisted representations are involved, we use 
vector spaces or tensor products over fields of rational fractions (it will not be specified as there is no risk of 
confusion).

\begin{defi} The Borel algebra $\U_q(\bo)$ is 
the subalgebra of $\U_q(\Glie)$ generated by $e_i$ 
and $k_i^{\pm 1}$ with $0\le i\le n$. 
\end{defi}
The Borel algebra is a Hopf subalgebra of $\U_q(\Glie)$ and contains 
the Drinfeld generators $x_{i,m}^+$, $x_{i,r}^-$, $k_i^{\pm 1}$, $\phi_{i,r}^+$ 
where $i\in I$, $m \geq 0$ and $r > 0$. When $\Glie = \sw_2$, these
elements generate $\U_q(\mathfrak{b})$.

The Borel algebra $\U_q(\bo)$ contains the Cartan-Drinfeld subalgebra $\U_q(\Hlie)^+$.

Similarly, we have the opposite Borel subalgebra $\U_q(\bo^-)$ generated by the $f_i, k_i^{\pm 1}$ with $0\le i\le n$. 

Denote $\tb \subset \U_q(\bo)$ the subalgebra generated by $\{k_i^{\pm 1}\}_{i\in I}$. 
Set $\tb^\times=\bigl(\C^\times\bigr)^I$, and endow it with a group
structure by pointwise multiplication.  Consider the group morphism
$$\overline{\phantom{u}}:P_\Q \longrightarrow \tb^\times\text{ by setting }\overline{\omega_i}(j)=q_i^{\delta_{i,j}}.$$ 
We use the standard
partial ordering on $\tb^\times$:
\begin{align}
\omega\leq \omega' \quad \text{if $\omega \omega'^{-1}$ 
is a product of $\{\ga_i^{-1}\}_{i\in I}$}.
\label{partial}
\end{align}

\subsection{Category $\mathcal{O}$ for representations of Borel
  algebras}\label{catO}

For a $\U_q(\mathfrak{b})$-module $V$ and $\omega\in \tb^\times$, we set
\begin{align}
V_{\omega}=\{v\in V \mid  k_i\, v = \omega(i) v\ (\forall i\in I)\}\,,
\label{wtsp}
\end{align}
and call it the weight space of weight $\omega$. 

We say that $V$ is Cartan-diagonalizable 
if $V=\underset{\omega\in \tb^\times}{\bigoplus}V_{\omega}$.

For any $i\in I$, $r\in\ZZ$ we have 
$$\phi_{i,r}^\pm (V_\omega)\subset V_\omega \text{ and }x_{i,r}^\pm (V_{\omega}) \subset V_{\omega \ga_i^{\pm 1}}.$$

\begin{defi} A series $\Psib=(\Psi_{i, m})_{i\in I, m\geq 0}$ 
of complex numbers such that 
$\Psi_{i,0}\neq 0$ for all $i\in I$ 
is called an $\ell$-weight. 
\end{defi}

For such an $\ell$-weight, identifying $(\Psi_{i, m})_{m\geq 0}$ with its generating series, we shall write
\begin{align*}
\Psib = (\Psi_i(z))_{i\in I},
\quad
\Psi_i(z) = \underset{m\geq 0}{\sum} \Psi_{i,m} z^m.
\end{align*}
We obtain a group structure on the set of $\ell$-weights that we denote by $\tb^\times_\ell$. 

We have a surjective morphism of groups
$\varpi : \tb^\times_\ell\rightarrow \tb^\times$ given by 
$\varpi(\Psib) = (\Psi_i(0))_{i\in I}$.
In particular, we have a factorization of each $\ell$-weight
\begin{equation}\label{fact}\Psib = \varpi(\Psib) \widetilde{\Psib}\end{equation}
as a product of its constant part $\varpi(\Psib)$ by its normalized part $\widetilde{\Psib}$, so that the normalized part has a trivial constant part.

\begin{defi} A $\U_q(\mathfrak{b})$-module $V$ is said to be 
of highest $\ell$-weight 
$\Psib\in \tb^\times_\ell$ if there is $v\in V$ such that 
$V =\U_q(\mathfrak{b}).v$ 
and the following hold:
\begin{align*}
e_i\, v=0\quad (i\in I)\,,
\qquad 
\phi_{i,m}^+v=\Psi_{i, m}v\quad (i\in I,\ m\ge 0)\,.
\end{align*}
\end{defi}

The $\ell$-weight $\Psib\in \tb^\times_\ell$ is uniquely determined by $V$ and is called the highest $\ell$-weight of $V$. 
The vector $v$ is said to be a highest $\ell$-weight vector of $V$.

\begin{prop}\label{simple} 
For any $\Psib\in \tb^\times_\ell$, there exists a simple 
highest $\ell$-weight module $L(\Psib)$ of highest $\ell$-weight
$\Psib$. This module is unique up to isomorphism.
\end{prop}

The submodule of $L(\Psib)\otimes L(\Psib')$ generated by a tensor
product of highest $\ell$-weight vectors is of highest
$\ell$-weight $\Psib\Psib'$. Hence $L(\Psib\Psib')$ is a
subquotient of $L(\Psib)\otimes L(\Psib')$.

\begin{defi}\cite{HJ}
For $i\in I$ and $a\in\CC^\times$, let 
\begin{align}
L_{i,a}^\pm = L(\Psib_{i,a}^{\pm 1})
\quad \text{where}\quad 
(\Psib_{i,a}^{\pm 1})_j(z) = \begin{cases}
(1 - za)^{\pm 1} & (j=i)\,,\\
1 & (j\neq i)\,.\\
\end{cases} 
\label{fund-rep}
\end{align}
\end{defi}
The representation $L_{i,a}^+$ (resp. $L_{i,a}^-$) is called a positive (resp. negative)
prefundamental representation in the category $\mathcal{O}$.

\begin{defi}\label{oned}\cite{HJ}
For $\omega\in \tb^\times$, let 
$$[\omega] = L(\Psib_\omega)
\quad \text{where}\quad 
(\Psib_\omega)_i(z) = \omega(i) \quad (i\in I).$$
\end{defi}
Note that such a representation $[\omega]$ is one-dimensional. For $\lambda\in P$, we will use
the notation $[\lambda]$ for the representation
$[\overline{\lambda}]$.
For $\lambda\in \tb^\times$, we set 
$$D(\lambda )=
\{\omega\in \tb^\times \mid \omega\leq\lambda\}.$$ 
The following category $\mathcal{O}$ is introduced in \cite{HJ}, mimicking the definition for classical Kac-Moody 
algebra, but using the weight space decomposition for the underlying finite-type Lie algebra.

\begin{defi} A $\U_q(\mathfrak{b})$-module $V$ 
is said to be in category $\mathcal{O}$ if:

i) $V$ is Cartan-diagonalizable,

ii) for all $\omega\in \tb^\times$ we have 
$\dim (V_{\omega})<\infty$,

iii) there exist a finite number of elements 
$\lambda_1,\cdots,\lambda_s\in \tb^\times$ 
such that the weights of $V$ are in 
$\underset{j=1,\cdots, s}{\bigcup}D(\lambda_j)$.
\end{defi}
The category $\mathcal{O}$ is a monoidal category. 

\newcommand{\mfr}{\mathfrak{r}}

Let $\Psib\in \mfr$ be the subgroup of $\tb^\times_\ell$
consisting of $\Psib = (\Psi_i(z))_{i\in I}$ such that $\Psi_i(z)$ is rational for any $i\in I$.

\begin{thm}\label{class}\cite{HJ} Let $\Psib\in\tb^\times_\ell$. The simple module $L(\Psib)$ is in the category 
$\mathcal{O}$ if and only if $\Psib\in \mfr$.
\end{thm}

Let $\mathcal{E}$ 
be the additive group of maps $c : P_\Q \rightarrow \ZZ$ 
whose support $\{\omega\in P_\Q,c(\omega) \neq 0\}$ is contained in 
a finite union of sets of the form $D(\mu)$. 

For $V$ in the category $\mathcal{O}$ we define the character of $V$ to be
an element of $\mathcal{E}$
\begin{align}
\chi(V) = \sum_{\omega\in\tb^\times} 
\text{dim}(V_\omega) [\omega]\,,
\label{ch}
\end{align}
where for $\omega\in P_\Q$, we have set $[\omega] = \delta_{\omega,.}\in\mathcal{E}$. 

This is coherent with the notation in Definition \ref{oned}, as for a one-dimensional representation $[\omega]$ therein, we have
$\chi([\omega]) = [\omega]$ in $\mathcal{E}$.

As for the category $\mathcal{O}$ of a classical Kac--Moody algebra,
the multiplicity of a simple module in a module of our category
$\mathcal{O}$ is well-defined (see \cite[Section 9.6]{kac}) and we have
the corresponding Grothendieck ring $K_0(\mathcal{O})$ (see \cite[Section 3.2]{HL}). Its elements are the formal
sums
$$\chi = \sum_{\Psib\in \mfr} \lambda_{\Psib} [L(\Psib)]$$ 
where the $\lambda_{\Psib}\in\ZZ$ are set so that $\sum_{\Psib\in
  \mfr, \omega\in P_\Q} |\lambda_{\Psib}|
\text{dim}((L(\Psib))_\omega) [\omega]$
is in $\mathcal{E}$.



\subsection{Finite-dimensional representations}\label{fdrep}


For $i\in I$ and $a\in\CC^*$, consider 
$$Y_{i,a} = \overline{\omega_i} \Psib_{i,aq_i}^{-1}\Psib_{i,aq_i^{-1}}.$$
If $M$ is a product of such $\ell$-weight, then $L(M)$ is finite-dimensional.  Moreover, the action of
$\U_q(\mathfrak{b})$ can be uniquely extended to an action of the full
quantum affine algebra $\U_q(\Glie)$, and any simple object in the
category $\mathcal{C}$  of (type $1$) finite-dimensional
representations of $\U_q(\Glie)$ is of this form. By \cite{cp} and \cite[Remark
3.11]{FH}, for $L$ a finite-dimensional module in the
  category $\mathcal{O}$, there is $M$ as above and
  $\omega\in\tb^\times$ such that 
	$$L \simeq L(M)\otimes [\omega].$$ 


\begin{ex}
For $i\in I$, $a\in\CC^\times$ and $k\geq 0$, we have the
Kirillov--Reshetikhin (KR) module
\begin{align}
W_{k,a}^{(i)} = L(Y_{i,a}Y_{i,aq_i^2}\cdots Y_{i,aq_i^{2(k-1)}})\,.
\label{KRmod}
\end{align}
The representations $V_i(a) = L(Y_{i,a})$ are called
fundamental representations.\end{ex}

\subsection{$\ell$-weight spaces}\label{qchar}

For a $\U_q(\mathfrak{b})$-module $V$ and $\Psib\in\tb_\ell^\times$, 
the linear subspace
\begin{align}
V_{\Psibs} =
\{v\in V\mid
\exists p\geq 0, \forall i\in I, 
\forall m\geq 0,  
(\phi_{i,m}^+ - \Psi_{i,m})^pv = 0\}
\label{l-wtsp} 
\end{align}
is called the $\ell$-weight space of $V$ of $\ell$-weight $\Psib$.

The study of these $\ell$-weight spaces is one of the motivations for the $q$-character theory \cite{Fre}.

A representation in the category $\mathcal{O}$ is the direct sum of its $\ell$-weight spaces. Moreover we have the following. 
\begin{thm}\cite{HJ} For $V$ in category $\mathcal{O}$, $V_{\Psib}\neq
  0$ implies $\Psib\in\mfr$.
\end{thm}

\begin{ex} (i) The fundamental representation $V_1(a)$ of $\U_q(\hat{sl_2})$ is $2$-dimensional and has $\ell$-weight spaces
attached respectively to $Y_{1,a}$ and $Y_{1,aq^2}^{-1}$.

(ii) It is proved in \cite{HJ, FH} that for $i\in I$ and $a\in \CC^*$, the $\ell$-weights of $L_{i,a}^+$ are all of the form $\Psib_{i,a}\overline{\omega}$ where $\omega \in Q$. In the $sl_2$-case, all $\ell$-weight spaces are of dimension $1$ and the $\ell$-weights are the $\Psib_{1,a} \overline{-2r\omega_1}$, $r\geq 0$.
\end{ex}

Let $\mathcal{E}_\ell$ be the additive group of maps
$c : \mfr\rightarrow \ZZ$  
such that 
$$\varpi(\{\Psib \in  \mfr \mid c(\Psib) \not = 0\})$$ 
is contained in a finite union of sets of the form $D(\mu)$, and such that 
for every $\omega\in P_\Q$, the set of $\Psib \in  \mfr$ satisfying $c(\Psib) \not = 0$ and 
$\varpi(\Psib) = \omega$ is finite. The map $\varpi$ is naturally extended to a surjective homomorphism 
$$\varpi : \mathcal{E}_\ell\rightarrow \mathcal{E}.$$
For $\Psib\in\mfr$, we define $[\Psib] = \delta_{\Psibs,.}\in\mathcal{E}_\ell$.

For $V$ in the category $\mathcal{O}$, we define \cite{Fre, HJ} the $q$-character of $V$ as
$$\chi_q(V) = 
\sum_{\Psibs\in\mfr}  
\mathrm{dim}(V_{\Psibs}) [\Psib]\in \mathcal{E}_\ell\,.$$


  




Following \cite{Fre}, we will use for $i\in I$, $a\in\mathbb{C}^*$ the $\ell$-weight $A_{i,a}$ which is set to be
$$Y_{i,aq_i^{-1}}Y_{i,aq_i}
\Bigl(\prod_{\{j\in I|C_{j,i} = -1\}}Y_{j,a}
\prod_{\{j\in I|C_{j,i} = -2\}}Y_{j,aq^{-1}}Y_{j,aq}
\prod_{\{j\in I|C_{j,i} =
-3\}}Y_{j,aq^{-2}}Y_{j,a}Y_{j,aq^2}\Bigr)^{-1}.$$

\begin{ex}
{\rm
In the case $\Glie = \widehat{\sw}_2$, we have:
\[
\chi_q(L_{1,a}^+) = [(1 - za)]\sum_{r\geq 0} [-2r\omega_1]\text{ , }\chi_q(L_{1,a}^-) = \left[\frac{1}{(1-za)}\right]\sum_{r\geq 0}
  A_{1,a}^{-1}A_{1,aq^{-2}}^{-1}\cdots A_{1,aq^{-2(r-1)}}^{-1}. 
\]
}
\end{ex}


Recall the factorization of $\ell$-weights (\ref{fact}). We will use the following.

\begin{prop}\label{caract} Let $L(\Psib)$ finite-dimensional and $\Psib'$ be an $\ell$-weight of $L(\Psib)$. 
Then its constant part $\varpi(\Psib')$ is uniquely determined by its normalized part $\widetilde{\Psib'}$.
\end{prop}

Note that this statement is not satisfied in general, for example it is not satisfied by positive prefundamental representations.

\begin{proof}
The $\ell$-weights of $L(\Psib)$ are of the form \cite{Fre, Fre2} :
\begin{equation}\label{infol}\Psib' = \Psib A_{i_1,a_1}^{-1}\cdots A_{i_N,a_N}^{-1}\end{equation}
where the $i_1,\cdots, i_N\in I$ and $a_1,\cdots, a_N\in \CC^*$. In particular 
$$\widetilde{\Psib'} = \widetilde{\Psib}\widetilde{A_{i_1,a_1}}^{-1}\cdots \widetilde{A_{i_N,a_N}}^{-1},$$
$$\varpi({\Psib'}) = \varpi(\Psib)\overline{-\alpha_{i_1} -\cdots -\alpha_{i_N}}.$$
But the $\widetilde{A_{i,a}}$ are free in the multiplicative group of $\ell$-weights, so 
$\varpi({\Psib'})$ is uniquely determined from $\widetilde{\Psib'}$.
\end{proof}

The algebra $\U_q(\Glie)$ has a natural $Q$-grading defined by
$$\deg\left(x^\pm_{i,m}\right) = \pm\a_i\text{ , }
\deg\left(h_{i,r}\right)=\deg\left(k_i^\pm\right) = \deg\left(c^{\pm 1/2}\right) = 0.$$

Let $\tilde{\U}_q^+(\Glie)$ (resp. $\tilde{\U}_q^-(\Glie)$) be the subalgebra of $\U_q(\Glie)$
consisting of elements of positive 
(resp. negative) $Q$-degree. These subalgebras should not be confused with the subalgebras $\U_q^\pm(\Glie)$ previously defined in terms of Drinfeld generators. Let 
$$X^+ = \sum_{j\in I, m\in\ZZ}\CC x_{j,m}^+\subset \tilde{\U}_q^+(\Glie).$$

\begin{thm}\label{apco}\cite{da} Let $i\in I$, $r > 0$, $m\in\ZZ$. We have
\begin{equation}\label{h}\Delta\left(h_{i,r}\right) \in h_{i,r}\otimes 1 + 1 \otimes h_{i,r} +  \tilde{\U}_q^-(\Glie) \otimes \tilde{\U}_q^+(\Glie),\end{equation}
\begin{equation}\label{x}\Delta\left(x_{i,m}^+\right) \in x_{i,m}^+\otimes 1 + \U_q(\Glie)\otimes \left(\U_q(\Glie) X^+\right).\end{equation}
\end{thm}

By definition, the $q$-character and the decomposition in $\ell$-weight spaces of a representation in the category $\mathcal{O}$ is determined by the action of $\U_q(\Hlie)^+$ \cite{Fre}. Therefore one can define the $q$-character $\chi_q(W)$ of a $\U_q(\Hlie)^+$-submodule $W$ of an object in the category $\mathcal{O}$.
 
The following result describes a condition on the $\ell$-weight of a linear combination of pure tensor products of weight vectors. It was originally proved in 
\cite{h4} for finite-dimensional representations in the category $\mathcal{C}$, but the proof is the same for general representations in the category $\mathcal{O}$.

\begin{thm}\cite{h4}\label{prodlweight}
Let $V_1, V_2$ representations in the category $\mathcal{O}$ and consider an $\ell$-weight vector
$$w = \left(\sum_{\alpha} w_\alpha\otimes v_\alpha\right) + \left(\sum_\beta w_\beta'\otimes v_\beta'\right)\in V_1\otimes V_2$$ 
satisfying the following conditions.

(i) The $v_\alpha$ are $\ell$-weight vectors of weight $\omega_\alpha$ and the $v_\beta'$ are weight vectors of weight $\omega_{\beta}$.

(ii) For any $\beta$, there is an $\alpha$ satisfying $\omega_\beta > \omega_\alpha$.

(iii) For $\omega\in\{\omega_\alpha\}_\alpha$, we have $\sum_{\left\{\alpha|\omega_\alpha = \omega\right\}} w_{\alpha}\otimes v_{\alpha}\neq 0$. 

\noindent Then the $\ell$-weight of $w$ is the product of the $\ell$-weight of one of the $v_\alpha$ by an $\ell$-weight of $V_1$.
\end{thm}

This result is one of the motivations for the constructions in this paper. It can be seen as an algebraic analog of the topological triangularity 
for $K$-theoretic stable map, as discussed in the introduction. It is also crucial for the results in \cite{hrec} about modules of highest $\ell$-weight.

\begin{ex}\label{exh4} Let $a, b\in\mathbb{C}^*$ and $V_1(a) \otimes V_1(b)$ a tensor product of fundamental representations of $\mathcal{U}_q(\hat{sl}_2)$. We denote by $v_a^\pm$ (resp. $v_b^\pm$) a natural basis of 
$\ell$-weight vectors of $V_1(a)$ (resp. $V_1(b)$). Then $v_a^-\otimes v_b^+$ and $(b-a)(v_a^+\otimes v_b^-) + a(q - q^{-1})(v_a^-\otimes v_b^+)$
are $\ell$-weight vectors of respective $\ell$-weights $Y_{1,aq^2}^{-1}Y_{1,b}$, $Y_{1,a}Y_{1,bq^2}^{-1}$, but not $v_a^+\otimes v_b^-$.
See \cite[Example 3.3]{h4} for details.
\end{ex}

\section{Algebraic stable maps}\label{sm}

In this section we define and construct algebraic stable maps on tensor products of modules in the category $\mathcal{O}$ (Definition \ref{defis}). We prove they define linear isomorphisms (Proposition \ref{isom}). We introduce the deformations of algebraic stable maps and we establish the rationality in the spectral parameter (Theorem \ref{ratios}). Then we give various examples in Subsection \ref{exsldeux}.

The main motivations for the constructions in this section are the stable maps and the triangularity of 
$\ell$-weight vectors in Theorem \ref{prodlweight} (see the Introduction).

\subsection{Definition and construction}

Motivated by Theorem \ref{prodlweight}, let us consider the following partial ordering on $P_{\mathbb{Q}}\times P_{\mathbb{Q}}$ which will be crucial in the following : 
\begin{equation}\label{crosordering}(\omega_1,\omega_2)\preceq (\omega_1',\omega_2')\text{ if and only if }(\omega_1+\omega_2 = \omega_1' + \omega_2'\text{ and }\omega_1 \leq \omega_1').\end{equation}
Obviously, this is equivalent to $(\omega_1+\omega_2 = \omega_1' + \omega_2'\text{ and }\omega_2\succeq \omega_2')$. 

Let $V$ and $W$ representations in the category $\mathcal{O}$.

For $v\in V$, $w\in W$ $\ell$-weight vectors of respective $\ell$-weights $\Psib$, $\Psib'$ and corresponding weights $\omega_1$, $\omega_2$, let us denote
$$(v\otimes w)_{\prec} = \sum_{(\varpi(\Psib),\varpi(\Psib'))\succ (\omega_1,\omega_2)} V_{\omega_1}\otimes W_{\omega_2},$$
$$(v\otimes w)_{\preceq} = v\otimes w + (v\otimes w)_{\prec}.$$

\begin{prop}\label{proplw} There is an $\ell$-weight vector of $V\otimes W$ in 
$(v \otimes w)_{\preceq}$. The $\ell$-weight of such an $\ell$-weight vector is $\Psib \Psib'$.
\end{prop}

\begin{proof} Let $M$ be the $\U_q(\Hlie)^+$-submodule of $V\otimes W$ generated by $(v\otimes w)_{\preceq}$. By the coproduct 
approximation formula (\ref{h}), we have the $\U_q(\Hlie)^+$-submodules
$$(v\otimes w)_{\prec} \subset M  = (v\otimes w)_{\prec} + \U_q(\Hlie)^+ .(v\otimes w) \subset V\otimes W.$$
Then
$$\chi_q(M) = \chi_q((v\otimes w)_{\prec}) + \chi_q(M/(v\otimes w)_{\prec}).$$
By coproduct formulas (\ref{h}) again, all weight vectors in $M/(v\otimes w)_{\prec}$ have the same $\ell$-weight $\Psib \Psib'$, and so one has
$$\chi_q(M/(v\otimes w)_{\prec}) = \text{dim}(M/(v\otimes w)_{\prec}) [\Psib \Psib'].$$
This implies
$$M = (v\otimes w)_{\prec} \oplus M'$$
where $M'\subset M_{\Psib\Psib'}$ is a space of $\ell$-weight vectors. Consider the component in $M'$ of the decomposition of $v\otimes w$ in this direct sum. 
It satisfies the properties in the statement.
\end{proof}

\begin{ex} In Example \ref{exh4} above, if $a\neq b$ we have the $\ell$-weight vector
$$(v_a^+\otimes v_b^-) + \frac{a(q - q^{-1})}{b - a}(v_a^-\otimes v_b^+) \in (v_a^+\otimes v_b^-)_{\preceq}.$$
\end{ex}

\begin{rem} In general the $\ell$-weight vector is not unique, even if $V$ and $W$ are simple. For example, in the $sl_2$-case, consider the tensor square $V^{\otimes 2}$ of a $2$-dimensional fundamental representation. Then the weight space $(V^{\otimes 2})_0$ is an $\ell$-weight space. For $v\in V_{-\omega}$, $w\in V_\omega$ non zero, then 
$$v\otimes w + V_{\omega} \otimes V_{-\omega}$$ 
is an affine subspace of dimension $1$ contained in an $\ell$-weight space.\end{rem}

Let us go back to the general case of $V$, $W$ in the category $\mathcal{O}$. Now we introduce a specific $\ell$-weight vector associated to $v\otimes w$.

As any object in the category $\mathcal{O}$, the representation $V\otimes W$ can be decomposed into a direct sum $\ell$-weight spaces. We have a corresponding projection 
$$\pi = \pi_{\Psib\Psib'} : V\otimes W\rightarrow (V\otimes W)_{\Psib\Psib'}$$ 
of $V\otimes W$ on the $\ell$-weight space associated to $\Psib\Psib'$. 

\begin{prop}\label{uniquelv} The $\ell$-weight vector $\pi(v\otimes w)$ is non-zero and 
$$\pi(v\otimes w)\in (v\otimes w)_{\preceq}.$$\end{prop}

\begin{proof} This is a direct consequence of the proof of Proposition \ref{proplw} as the 
projection $\pi(v\otimes w)$ is the $\ell$-weight vector constructed there in which is non-zero : 
$$\pi((v\otimes w)_{\prec})\subset (v\otimes w)_{\prec},$$ 
and
$$\pi(v\otimes w) \in v\otimes w + (v\otimes w)_\prec\subset M\setminus (v\otimes w)_{\prec}.$$ 
Hence the result.\end{proof}

\begin{defi}\label{defis} We define the algebraic stable map
$$S_{V,W} : V\otimes W\rightarrow V\otimes W$$
 by 
$$S_{V,W} = \pi_{\Psib\Psib'}\text{ on }(V)_{\Psib}\otimes (W)_{\Psib'}.$$
\end{defi}

\begin{prop}\label{isom} $S_{V,W}$ is a linear isomorphism.\end{prop}

\begin{proof} Let us decompose $V\otimes W$
as a direct sum of tensor products of weight spaces of $V$ and $W$. Then the partial ordering $\prec$ on $P_{\mathbb{Q}}\times P_{\mathbb{Q}}$ induces
a filtration on $V\otimes W$ by $\mathcal{U}_q(\Hlie)^+$-submodules. It follows from  Proposition \ref{uniquelv} that $S_{V,W}$ is compatible 
with the filtration and that it induces the identity on the corresponding graded space. This implies the injectivity. We get the result as
the weight spaces of $V\otimes W$ are finite dimensional and stable by $S_{V,W}$.
\end{proof}

\subsection{Deformations}

In the following, it will be useful to consider deformations of algebraic stable maps which depend on a spectral parameter $u$.

As the subalgebra $\U_{q,u}(\bo) = \U_q(\bo)\otimes \CC(u)\subset \U_{q,u}(\Glie)$ is stable by the automorphism $\tau_u$ considered in Section \ref{debut}, 
for a formal parameter $u$ and $V$ a representation in the category $\mathcal{O}$, we have the deformed $\U_{q,u}(\bo)$-module $V(u)$ as above.
The representation $V(u)$ has also a decomposition into a direct of $\ell$-weight spaces, and the $\ell$-weights are formal power series with 
coefficients in $\mathbb{C}[u]$.

\begin{rem}\label{prerem}  The same proof as in Proposition \ref{proplw} gives an analog $\ell$-weight vector in the tensor product  $V(u)\otimes W$, that is when $V$ is replaced by the deformed module $V(u)$. Its $\ell$-weight is $\Psib(u)\Psib'$ where $\Psib(u)$ is defined as the $\ell$-weight 
$$\Psib(u) : z\mapsto \Psib(zu).$$
\end{rem}

\begin{prop}\label{uni} Let $V$, $W$ simple modules in the category $\mathcal{O}$ such that $V$ or $W$ is finite-dimensional. Then the $\ell$-weight vector of $\ell$-weight $\Psib(u)\Psib'$ in $(v\otimes w)_\preceq \subset V(u)\otimes W$ is unique.
\end{prop}

\begin{proof} Suppose first that $V$ is finite-dimensional. It suffices to prove 
that $(v\otimes w)_\prec$ does not contain any $\ell$-weight vector $X$ of $\ell$-weight $\Psib(u)\Psib'$. 
In a decomposition of such an $\ell$-weight vector $X$
as a sum of tensor products of $\ell$-weight vectors, a term of the form $v'\otimes w'$ would occur, 
with $v'$ of $\ell$-weight $\Psib(u)$. Then it follows from Proposition \ref{caract} for 
the finite-dimensional representation $V$ that the weight of $v'$ is the weight of $v$. 
This contradicts $X\in (v\otimes w)_\prec$.

This is analog in the case $W$ is finite-dimensional as
$$(V(u)\otimes W)(u^{-1}) \simeq V\otimes W(u^{-1}).$$
\end{proof}

By Remark \ref{prerem}, for formal parameters $u, v$, the map $S_{V,W}$ can be deformed by replacing the
representations $V$, $W$ respectively by $V(u)$ and $W(v)$. We get the deformed algebraic stable map
$$S_{V,W}(u,v) : V(u)\otimes W(v)\rightarrow V(u)\otimes W(v).$$

\begin{thm}\label{ratios} $S_{V,W}(u,v)$ depends only on the quotient $u/v$ and is rational in this parameter.
\end{thm}

In the following, it will just be denoted by $S_{V,W}(u/v)$.

\begin{proof}  The first point follows from the elementary observation :
$$V(u)\otimes W(v)\simeq (V(u/v)\otimes W(1))(v).$$
Then an $\ell$-weight vector in $V(u/v)\otimes W(1)$ is still an $\ell$-weight vector when the action is twisted by $v$.

So for the second point, we can suppose $v = 1$. Let us consider a non zero weight space 
$$(V(u)\otimes W)_\lambda.$$ 
As it is finite-dimensional, there is a finite $N > 0$ such that the $\ell$-weight vectors in $(V(u)\otimes W)_\lambda$ are uniquely determined by the action of the $h_{i,m}$, $i\in I$, $0 < m \leq N$. By the definition of the coproduct, for an element $x\in \U_q(\bo)$ of degree $m\geq 0$, $\Delta(x)$ is a sum of pure tensors $a\otimes b$ with $a,b$ of degree $\leq m$. So the $h_{i,m}$ with $0 < m\leq N$ act on $V(u)\otimes W$ as polynomials in $u$ of degree lower than $M$. Besides, consider a basis of $V(u)\otimes W$ of pure tensor of $\ell$-weight vectors. Then there is a partial ordering on such a basis induced from $\preceq$. Indeed, for $v$, $v'$, $w$, $w'$ $\ell$-weight vectors of respective weights $\omega_1$, $\omega_1'$, $\omega_2$, $\omega_2'$, we set
$$v\otimes w \preceq v'\otimes w'\text{ if }(\omega_1,\omega_2)\preceq (\omega_1',\omega_2').$$ 
We can re-order the basis of $\ell$-weight vectors so that it is compatible with $\preceq$. Then the action of the $h_{i,m}$ gives triangular matrices in such a basis thanks to the coproduct formula (\ref{h}). So the operators $h_{i,m}$ are pseudo-diagonalizable on $V(u)\otimes W$ over the field $\CC(u)$. Hence the projection on the corresponding generalized eigenspaces are rational.
\end{proof}

\subsection{Examples}\label{exsldeux}
We consider various explicit examples of the maps constructed in the previous sections. Our examples contain some infinite-dimensional representations.

\begin{ex}\label{eval} Let $\Glie = \hat{sl}_2$ and for $k$ and consider the evaluation representation 
$$W_k = L(Y_{1,q^{-1}}Y_{1,q^{-3}}\cdots Y_{1, q^{1 - 2k}}),$$
see \cite[Section 4.1]{HJ}.
We can choose a basis of $\ell$-weight vectors $v_0,\cdots, v_k$ of $W_k$ so that $e_1.v_i = v_{i-1}$ for $i\geq 1$.
Then we have 
$$e_0.v_i = q^{2 - k}[i+1]_q [k - i]_q v_{i+1}\text{ and }h_{1,1}.v_i = (q^{2-2i} + q^{-2i} -q^{-2k} - q^2)(q - q^{-1})^{-1}v_i.$$
The spectrum of $h_{1,1} = q^{-2} e_1e_0 - e_0e_1$ is simple and so the action of $h_{1,1}$ is sufficient 
to determine the $\ell$-weight vectors of $W_k$. For $u$ a formal variable, this is also true on a tensor product 
 $W_k(u)\otimes W_l$ :
$$h_{1,1}.(v_i\otimes v_j) 
= q^{-2}e_1(e_0v_i\otimes v_j + q^{2i - k}v_i\otimes e_0v_j) - e_0(e_1v_i\otimes v_j + q^{k - 2i}v_i\otimes e_1v_j)$$
$$= (h_{1,1}.v_i\otimes v_j) + (v_i\otimes h_{1,1}.v_j) $$
$$+ q^{-4 + k - 2i}e_0v_i\otimes e_1v_j + q^{-2 + 2i - k}e_1v_i\otimes e_0v_j - q^{2i-k-2}e_1v_i\otimes e_0v_j - q^{k-2i}e_0v_i\otimes e_1v_j$$
$$= (h_{1,1}.v_i\otimes v_j) + (v_i\otimes h_{1,1}.v_j) + (q^{-4} - 1) q^{k - 2i} (e_0.v_i \otimes v_{j-1})$$
$$ = P_{i,j}^{k,l}(u) (v_i\otimes v_j) + u \alpha_{i,k} (q + q^{-1}) (v_{i+1}\otimes v_{j-1}),$$
where $\alpha_{i,k} = (q^{-1} - q) q^{-2i}[i+1]_q[k-i]_q$ and the
$$ P_{i,j}^{k,l}(u) = ((q^2 + 1)q^{-2j} - q^2 - q^{-2l} + u ((q^2 + 1)q^{-2i} - q^2 - q^{-2k}))(q- q^{-1})^{-1}.$$
are the eigenvalues of $h_{1,1}$ on the tensor product.
Then we get an $\ell$-weight vector of the form
$$v_i\otimes v_j + \frac{u\alpha_{i,k}(q + q^{-1})v_{i+1}\otimes v_{j-1}}{P_{i,j}^{k,l}(u) - P_{i+1,j-1}^{k,l}(u)}  + \frac{u^2\alpha_{i,k}\alpha_{i+1,k}(q + q^{-1})^2v_{i+2}\otimes v_{j-2}}{(P_{i,j}^{k,l}(u) - P_{i+1,j-1}^{k,l}(u))(P_{i,j}^{k,l}(u) - P_{i+2,j-2}^{k,l}(u))}  + \cdots$$
So by substituting $v_i\in W_k(u)$ by $v_i' = (\alpha_{0,k}\alpha_{1,k}\cdots \alpha_{i-1,k})v_i$, we get the following :
$$S_{W_k,W_l}(u).(v_i'\otimes v_j) = v_i'\otimes v_j + \sum_{0 < \lambda \leq \text{Min}(j,k-i)} \frac{u^\lambda v_{i+\lambda}'\otimes v_{j-\lambda}}{[\lambda]_q!(uq^{-2i} - q^{2(1-j)})\cdots (uq^{1-2i -\lambda} - q^{1 -2j + \lambda})}.$$
In particular, for $k = l = 1$, $v_1' = (q^{-1} - q)v_1$ and we get :
$$S_{W_1,W_1}(u).(v_1\otimes v_0) = v_1\otimes v_0 \text{ , }S_{W_1,W_1}(u) .(v_0\otimes v_1) = v_0\otimes v_1 + \frac{u(q - q^{-1})}{1 - u}v_1\otimes v_0.$$
This matches the $\ell$-weight vectors computed in Example \ref{exh4}.
\end{ex}

\begin{rem}\label{rconj}In the example above, one can observe the following : $S_{W_k,W_l}(u)$ is regular at $u = 0$ and at $u = \infty$, and $S_{W_k,W_l}(0) = \text{Id}$. This is not true in
general as the examples below will show.
\end{rem}

\begin{ex} Let $\Glie = \hat{sl}_2$ and $k \geq 0$. 
We can choose a basis of $\ell$-weight vectors $w_0, w_1 \cdots$ of the 
prefundamental representation $L_1^-$ so that $e_1.w_i = w_{i-1}$ for $i\geq 1$, see \cite[Section 4.1]{HJ}. Then
$$h_{1,1}. w_i = \frac{q^{2-2i} + q^{-2i}  - q^2}{q - q^{-1}} w_i.$$
Then as above one has :
$$S_{W_k,L_1^-}(u).(v_i'\otimes w_j) = 
v_i'\otimes w_j + \sum_{0 < \lambda \leq \text{Min}(j,k-i)} \frac{u^\lambda v_{i+\lambda}'\otimes w_{j-\lambda}}
{[\lambda]_q!(uq^{-2i} - q^{2(1-j)})\cdots (uq^{1-2i -\lambda} - q^{1 -2j + \lambda})}$$
$$S_{L_1^-,W_l}(u).(w_i'\otimes v_j) = w_i'\otimes v_j +  \sum_{0 < \lambda \leq j} \frac{u^\lambda w_{i+\lambda}'\otimes v_{j-\lambda}}
{[\lambda]_q!(uq^{-2i} - q^{2(1-j)})\cdots (uq^{1-2i -\lambda} - q^{1 -2j + \lambda})}$$
where 
$w_i' = q^{-3i(i-1)/2}(-1)^i[i]_q!w_i$.
\end{ex}

\begin{ex}\label{forb} Let $\Glie = \hat{sl}_2$ and $k \geq 0$. 
We can choose a basis of $\ell$-weight vectors $z_0, z_1 \cdots$ of the 
prefundamental representation $L_1^+$ so that $e_1.z_i = z_{i-1}$ for $i\geq 1$, see \cite[Section 7.1]{HJ}.
Then we have $e_0.z_i = - q^{i+2}\frac{[i+1]_q}{q - q^{-1}} z_{i+1}$, $h_{1,1}.z_i = (q^{-1} - q)^{-1}z_i$ and
$$S_{W_k,L_1^+}(u).(v_i'\otimes z_j) = 
\sum_{0 \leq \lambda \leq \text{Min}(j,k-i)} \frac{q^{2i\lambda + \frac{\lambda(\lambda - 1)}{2}}}{[\lambda]_q!} v_{i+\lambda}'\otimes z_{j-\lambda},$$
$$S_{L_1^+,W_l}(u).(z_i'\otimes v_j) =  \sum_{0 \leq \lambda \leq j} \frac{q^{(2j - i)\lambda - \lambda(\lambda + 1)}}{[\lambda]_q!}(-u)^{\lambda} z_{i+\lambda}'\otimes v_{j-\lambda},$$
where 
$z_i' = [i]_q!z_i$.
\end{ex}

\begin{ex}\label{exl}
Let $\Glie = \hat{sl}_2$. 
$$S_{L_1^-,L_1^+}(u).(w_i'\otimes z_j) =
\sum_{0 \leq \lambda \leq j}  \frac{q^{2i\lambda + \frac{\lambda(\lambda - 1)}{2}}}{[\lambda]_q!} w_{i+\lambda}'\otimes z_{j-\lambda}$$
$$S_{L_1^+,L_1^-}(u).(z_i'\otimes w_j) = 
 \sum_{0 \leq \lambda \leq j} \frac{q^{(2j - i)\lambda - \lambda(\lambda + 1)}}{[\lambda]_q!}(-u)^{\lambda}
 z_{i+\lambda}'\otimes w_{j-\lambda},$$
$$S_{L_1^-,L_1^-}(u).(w_i'\otimes w_j) = w_i'\otimes w_j 
+\sum_{0 < \lambda \leq j} \frac{u^\lambda w_{i+\lambda}'\otimes w_{j-\lambda}}
{[\lambda]_q!(uq^{-2i} - q^{2(1-j)})\cdots (uq^{1-2i -\lambda} - q^{1 -2j + \lambda})}$$

$S_{L_1^+,L_1^+}(u) = \text{Id}$ as each weight space of $L_1^+(u)\otimes L_1^+(v)$ is an $\ell$-weight space.


\end{ex}

\subsection{Normalized stable map}\label{normsp}

Note that $S_{V,W}(u)$ may have poles and does not necessarily converges to $S_{V,W}$ when $u\rightarrow 1$.
However, it is possible to define a renormalized limit. Indeed, there is unique
$N\in \ZZ$ such that the limit
$$\text{lim}_{u\rightarrow 1}(u - 1)^N S_{V,W}(u)$$
exists and is non zero. It is denoted
$$S_{V,W}^{norm} : V\otimes W\rightarrow V\otimes W.$$

\begin{rem} The normalized $S_{V,W}^{norm}$ is not equal to $S_{V,W}$ in general.
\end{rem}

\begin{ex} Let $\Glie = \hat{sl}_2$. We use the notations of Example \ref{eval}. For $k = l = 1$, we get :
$$S_{W_1,W_1}^{norm}.(v_1\otimes v_0) = v_1\otimes v_0 \text{ , }S_{W_1,W_1}^{norm} .(v_0\otimes v_1) = (q^{-1} - q)v_1\otimes v_0.$$
We observe that $S_{W_1,W_1}^{norm}$ is not invertible in opposition to $S_{W_1,W_1}$ (see Proposition \ref{isom}). 
In fact, $S_{W_1,W_1}$ is just the identity as the weight spaces of $W_1\otimes W_1$ are its $\ell$-weight spaces and 
so a tensor product of two $\ell$-weight vectors of $W_1$ lies in an $\ell$-weight space of $W_1\otimes W_1$.
\end{ex}

\section{$R$-matrices and finite-dimensional representations}\label{rel}

In the case of finite-dimensional representations, the algebraic stable maps are related to multiplications 
by factor of the universal $R$-matrix (Proposition \ref{relkt}), by analogy to the original stable maps. 
Note that this is not true in general for the category $\mathcal{O}$ as the action of the universal $R$-matrix is not always well-defined. 

The properties established in this section are at the origin of the present work and discussions with
A. Okounkov were crucial for its development (see in particular Remark \ref{remao}).

\subsection{Reminders on $R$-matrices}\label{interw}





Let $W$ and $W'$ be simple finite-dimensional representations of $\U_q(\Glie)$.



For $a\in\CC^*$ generic 
(that is in the complement of a finite set of $\CC^*$), we have an isomorphim of $\U_q(\Glie)$-modules 
$$\mathcal{R}_{W,W'}(a) : W\otimes W'(a)\rightarrow W'(a)\otimes W.$$
  
Considering $a$ as a variable $u$, we get a rational map in $u$
$$\mathcal{R}_{W,W'}(u) : (W\otimes W')\otimes\CC(z)\rightarrow (W'\otimes W)\otimes \CC(u).$$
This map is normalized so that for $v\in W$, $v'\in W'$ highest weight vectors,
we have 
$$(\mathcal{R}_{W,W'}(u))(v\otimes v') = v'\otimes v.$$ 
Note that $\mathcal{R}_{W,W'}(u)$ defines an isomorphism of representations of $\U_{q,z}(\Glie)$
As in section \ref{normsp}, we may consider the first term in the development in $u - 1$, and we
get a non zero morphism
$$\mathcal{R}_{W,W'} : W\otimes W' \rightarrow W'\otimes W$$
which is not necessary invertible (see \cite{Hbou} for references).

By original results of Drinfeld, it is well-known that these intertwiners come from the universal $R$-matrix 
$$\mathcal{R}(z)\in (\U_q(\Glie)\hat{\otimes} \U_q(\Glie))[[z]]$$ 
which is a solution of the Yang-Baxter equation (here $\hat{\otimes}$ is a slightly completed tensor product). 

Let $\mathcal{U}_q(\bo)^\pm = \mathcal{U}_q(\bo)\cap \mathcal{U}_q^\pm(\Glie)$ and $\mathcal{U}_q(\bo^-)^\pm = \mathcal{U}_q(\bo^-)\cap \mathcal{U}_q^\pm(\Glie)$.

The
universal $R$-matrix has a factorization \cite{kt, da}
$$\mathcal{R}(z) = \mathcal{R}^+(z)\mathcal{R}^0(z)\mathcal{R}^-(z)\mathcal{R}^\infty,$$
where $\mathcal{R}^\pm(z) \in \U_q(\mathfrak{b})^\pm\hat{\otimes}
\U_q(\mathfrak{b}^-)^\mp[[z]]$, 
$$\mathcal{R}^0(z) = \text{exp}\left( -  \sum_{m > 0,i,j\in I} \frac{(q_i - q_i^{-1})(q_j -
    q_j^{-1})m
    \tilde{B}_{i,j}(q^m)}{(q - q^{-1})[m]_q}z^mh_{i,m}\otimes h_{j,-m}\right)$$
with $\tilde{B}(q)$ the inverse of the symmetrized quantum Cartan matrix of $\overline{\mathfrak{g}}$, 
and $\mathcal{R}^\infty = q^{-t_\infty}$ where
$t_\infty\in\overline{\mathfrak{h}}\otimes\overline{\mathfrak{h}}$ is the
canonical element (for the standard invariant symmetric bilinear form
as in \cite{da}), that is, if we denote formally $q = e^h$, then
$$\mathcal{R}^\infty = e^{-h t_\infty}.$$ 
Hence if $k_i.x = q^{(\lambda,\alpha_i)}x$ and $k_i.y = q^{(\mu,\alpha_i)}y$, 
then $\mathcal{R}^\infty.(x\otimes y) = q^{-(\lambda,\mu)}x\otimes y$.

The following is a direct consequence of well-known results. The algebras $\tilde{\U}_q^\pm(\Glie)$ are defined in section \ref{qchar}.

\begin{prop} We have
$$\mathcal{R}^+(z)\in 1 + (\tilde{\U}_q^+(\Glie) \otimes \tilde{\U}_q^-(\Glie))[[z]].$$
\end{prop}

\begin{proof} Let us recall that for a variable $x$, the $q$-exponential in $x$ is a formal power series
$\text{exp}_{q^p} (x) = \sum_{r\geq 0} \frac{x^r}{[r]_{q^p}'!}$ where
$p\in\ZZ$ and $[r]_v'! = \prod_{1\leq s\leq r}\frac{v^{2s} - 1}{v^2 -
  1} = v^{\frac{r(r-1)}{2}}[r]_v!$ for $r\geq 0$.
	
	 Let us also remind \cite{bec, da} that we have the root vectors
$E_\alpha\in \U_q(\bo)$, $F_\alpha\in \U_{q}(\bo^-)$ for
$$\alpha\in \Phi_+^{Re} = \Phi_0^+\sqcup\{\beta +m\delta|m > 0, \beta \in \Phi_0\}.$$
Here $\Phi_0$ (resp. $\Phi_0^+$) is the set of roots (resp. positive
roots) of $\overline{\mathfrak{g}}$ and $\delta$ is the standard imaginary
root of $\mathfrak{g}$.

$\mathcal{R}^+$ (resp. $\mathcal{R}^-$) is a product of
$q$-exponentials of a scalar multiple of a tensor product of root vectors $z^m E_{\alpha + m\delta}\otimes
F_{\alpha + m\delta}$ with $m \geq 0$, $\alpha\in\Phi_0^+$ (resp. with
$m > 0$, $\alpha\in\Phi_0^-$). The result follows.
\end{proof}


\begin{ex} In the case $\mathfrak{g} = \hat{sl_2}$, we have an explicit description in terms of Drinfeld generators : 
$$\mathcal{R}^+(z) = \overset{\rightarrow}{\prod}_{m\geq 0}\text{exp}_q\left((q^{-1} -
  q)z^m x_{1,m}^+\otimes x_{1,-m}^-\right),$$ 
	$$\mathcal{R}^-(z) =
\overset{\leftarrow}{\prod}_{m > 0}\text{exp}_{q^{-1}}\left((q^{-1} - q)z^m k_1^{-1}x_{1,m}^-\otimes
  x_{1,-m}^+  k_1\right),$$
$$\mathcal{R}^0(z) = \text{exp}\left( - (q - q^{-1}) \sum_{m > 0} z^m 
  \frac{m}{[m]_q(q^m + q^{-m})}h_{1,m}\otimes h_{1,-m}\right).$$
\end{ex}

\begin{rem} In the formula
\footnote{The sign in the $q$-exponential for $\mathcal{R}^-(z)$ in \cite[Example 7.1]{FH} has to be changed as above. Then, $q - q^{-1}$ should be $q^{-1} - q$ in \cite[Example 7.8]{FH} (for $L_V^-(z)$ and $t_V(z,u)$) and in \cite[Section 5.7]{FH} (for $\mathcal{L}_V^-(z)$). In the first line for the image of the transfer-matrix, $-(q - q^{-1})^2$ should be $(q - q^{-1})^2$.} given in \cite[Example 7.1]{FH}, the products defining $\mathcal{R}^+(z)$ and $\mathcal{R}^-(z)$ are
not ordered as the ordering has no importance for the examples considered in that paper 
(indeed on the representation $R_{1,1}$ there, the $x_m^+$ and $x_{m+1}^-$ act by $0$ for $m > 0$). 
In general we have to use the ordering given above, which is obtained from the convex ordering on affine roots
$$\alpha < \alpha + \delta < \alpha + 2\delta < \cdots < \delta < 2 \delta < \cdots < - \alpha +2\delta < - \alpha + \delta.$$
\end{rem}

For $\tau$ the twist we have 
$$\tau \circ \Delta = \mathcal{R}(u)\Delta \mathcal{R}^{-1}(u)$$ 
and so the specializations of $\tau \circ \mathcal{R}(z)$ give morphisms of $\U_q(\Glie)$-modules.

\subsection{Relations to known constructions}

The main motivation for this work in the theory of stable envelopes by Maulik-Okounkov \cite{mo} (see the Introduction). In the case of
finite-dimensional fundamental representations $V$, $W$ of a simply-laced quantum affine algebra, it provides the geometric construction of maps
$$\text{Stab}_{V,W}^\pm(z) : V(z)\otimes W\rightarrow V(z)\otimes W,$$
the chambers being defined from the quotient of spectral parameters (see Example \ref{exfam}).
These maps depend on additional parameters (see the Introduction), and it is expected, and proved in certain cases \cite[Section 2.3.3]{os}, that some 
choice of these parameters, $\text{Stab}_{V,W}^\pm(z)$ is obtained from $\mathcal{R}^\pm(z)$ multiplied by a factor in $(\U_q(\Hlie)^+\otimes \U_q(\Hlie)^-)[[z]]$. 

By analogy, we have the following for $V$, $W$ in the category $\mathcal{O}$ such that $W$ is simple 
finite-dimensional (not necessarily fundamental). In particular, when the statement above is established, 
the algebraic stable map $S_{V,W}(u)$ 
in this paper is the Maulik-Okounkov stable map $\tau \text{Stab}_{W,V}^+ \tau$, up to a factor which is a tensor product of Cartan-Drinfeld elements.

\begin{prop}\label{relkt} The multiplication by $\mathcal{R}^+(u)$ (resp. by $\mathcal{R}^-(u)$) on $V(u)\otimes W$ corresponds to the action of 
$$\tau\circ S_{W,V}(u^{-1}) \circ \tau\text{ (resp. }\mathcal{R}^\infty(S_{V,W}(u))^{-1}(\mathcal{R}^\infty)^{-1}).$$
\end{prop}

\begin{rem} The construction of $R$-matrices from the Cartan-Drinfeld subalgebra might be seen
as a reminiscent of the vertex algebra representations of quantum affine algebras \cite{FJ}.\end{rem}

\begin{proof} Let us explain it for $\mathcal{R}^+(u)$ (this is analog for $\mathcal{R}^-(u)$). $\mathcal{R}^+(u)$ is a product of $q$-exponentials 
in the form described in section \ref{interw}. Hence for $v\in V$, $w\in W$ $\ell$-weight vectors, we get 
$$(\tau\circ \mathcal{R}^+(u)\circ \tau).(w\otimes v)\in (w\otimes v)_{\preceq}.$$ 
Moreover it is known \cite[End of Section 5]{kt} 
that $\mathcal{R}^+(u)$ defines a morphism of $\U_q(\Hlie)^+$-modules as above
$$\tau\circ \mathcal{R}^+(u) : V(u)\otimes_d W\rightarrow W\otimes V(u).$$
Indeed, the Drinfeld coproduct is obtained from the usual coproduct by conjugation by $\mathcal{R}^+(u)$ \cite[Proposition 5.1]{kt} 
(see also \cite{ekp} for more details).

Consequently 
$$(\tau\circ \mathcal{R}^+(u)\circ \tau).(w\otimes v)$$ 
is an $\ell$-weight vector. 
Hence, by the uniqueness in Proposition \ref{uni}, this implies
$$S_{W,V}(u^{-1}) = \tau\circ \mathcal{R}^+(u)\circ \tau.$$
\end{proof}

\begin{rem}\label{remao} (i) The compatibility of $\mathcal{R}^\pm(u)$ with the actions of $\U_q(\Hlie)^+$ discussed in \cite{kt} was 
pointed out by A. Okounkov to the author as an answer to his question about the seeming compatibility between stable maps and $\ell$-weight vectors 
(see Example \ref{exh4}). 

(ii) Note that the statement of Proposition \ref{relkt} does not make sense in general in the category $\mathcal{O}$ as $\mathcal{R}^+(z)$ can not
be applied to $V\otimes W$ for general representations $V$, $W$ in this category $\mathcal{O}$.

(iii)
For $\Glie = \hat{sl_2}$, in the case of a prefundamental representation $V$ and a Kirillov-Reshetikhin module $W$, a construction is discussed in \cite{psz} in terms of equivariant $K$-theory.

(iv) $\mathcal{R}^\infty$ commutes with the twist $\tau$.
\end{rem}

It is well known the full $R$-matrix $\mathcal{R}(u)$ has a rational action up to a scalar. We deduce the following Corollary from our constructions.

\begin{cor}\label{ratfact} Let $V$, $W$ be finite-dimensional representations of $\U_q(\Glie)$. Then $\mathcal{R}^+(u)$ and $\mathcal{R}^-(u)$ define rational operators on $V(u)\otimes W$. The factor $\mathcal{R}^0(u)$ defines a rational operator up to a scalar factor.
\end{cor}

\begin{proof} The rationality of $\mathcal{R}^\pm(u)$ follow from the rationality of the algebraic stable maps and from Proposition \ref{relkt}. 
As the universal $R$-matrix defines a rational operator up to a scalar factor, we get the result for $\mathcal{R}^0(u)$.
\end{proof}


\begin{ex}\label{exfam} Let $\Glie = \hat{sl_2}$, $V \simeq W \simeq W_1$ fundamental representation. Then the action of 
$\mathcal{R}^+(u)$ (resp. of $\mathcal{R}^-(u)$) reduces to the action of 
$$1 + \sum_{m\geq 0}(q^{-1} - q) u^m x_m^+\otimes x_{-m}^-\text{ (resp. of }1 + \sum_{m > 0}u^m (q^{-1} - q)k^{-1}x_m^- \otimes x_{-m}^+ k).$$
In particular, in the basis $(v_0\otimes v_1, v_1\otimes v_0)$, we have
$$\mathcal{R}^+(u) = \begin{pmatrix}1 & \frac{q^{-1} - q}{1 - u}\\ 0 & 1\end{pmatrix} = \tau\circ S_{W,V}(u^{-1}) \circ \tau,$$
$$\mathcal{R}^-(u) = \begin{pmatrix}1 & 0\\
\frac{u(q^{-1} - q)}{1 - u} & 1\end{pmatrix} = (S_{V,W}(u))^{-1}.$$
Note that the action of $\mathcal{R}^\infty$ on the zero weight space of $V(u)\otimes W$ is the multiplication by $q^{(\omega,\omega)} = q^{\frac{1}{2}}$.

We recover the computations in \cite[Section 7.1.7]{os} in the case of  $X = T^*\mathbb{P}^1$ the cotangent bundle of $\mathbb{P}^1$ with the natural action of $T = A\times \mathbb{C}^*$ where $A = \mathbb{C}^*$. Indeed, the action of $A$ is induced from the action on $\mathbb{C}^2$ and given by characters denoted by $u_1$, $u_2$. 
The additional factor $\mathbb{C}^*$ has a non-trivial action on the fibers of $T^*\mathbb{P}^1$ given by the character $h$. The fixed points in $X$ are 
$$X^A = \{p_0 = [1:0], p_1 = [0:1]\}.$$ 
There are $2$ chambers $\mathcal{C}_\pm = \{x\in\mathbb{R}|\pm u(x) > 0\} \subset \mathfrak{a}_{\mathbb{R}}  = \mathbb{R}$ where $u = u_1/u_2$. Then
$$\{p_0\}\succ_{\mathcal{C}_-}\{p_1\}\text{ et }\{p_1\}\succ_{\mathcal{C}_+}\{p_0\}.$$
We have the basis $[p_0]$, $[p_1]$ of $K_T(X^A)$ over $\mathbb{C}(u,h^{\frac{1}{2}})$. Using the inclusion $X^A\subset X$ and the corresponding 
injection $i : K_T(X)\hookrightarrow K_T(X^A)$, we have a corresponding basis in $K_T(X)$. In these basis, for a choice 
of a slope (see footnote 1 and \cite[Section 7.1.5]{os} for details), we get the matrices in the basis of fixed points $([p_1], [p_0])$ (here we set $q = h^{-\frac{1}{2}}$) :
$$\text{Stab}_{\mathcal{C}_+} = \begin{pmatrix} q^{-1}  (u - q^2) & 0 \\ u (q^{-1} - q) &  1- u  \end{pmatrix} = 
\begin{pmatrix} q^{-1}  (u - q^2) & 0 \\ 0 &  1- u  \end{pmatrix}
\begin{pmatrix} 1 & 0 \\ \frac{u (q^{-1} - q)}{1 - u} &  1  \end{pmatrix}
,$$
$$\text{Stab}_{\mathcal{C}_-} = \begin{pmatrix}u - 1   &  q^{-1} - q \\ 0 & q(q^{-2} - u)\end{pmatrix}
= \begin{pmatrix}u - 1   &  0 \\ 0 & q(q^{-2} - u)\end{pmatrix}  
\begin{pmatrix} 1   &  \frac{q^{-1} - q}{u - 1} \\ 0 & 1\end{pmatrix}.$$
The triangularity property can clearly be observed here. The diagonal factors correspond to the renormalization 
of $\ell$-weight vectors (see below).
\end{ex}

\begin{ex} 
Let us consider the tensor product $L_1^+\otimes W_k$ with the notations of the previous examples. Then the action
 on $L_1^+(u)\otimes W_k$ of $\mathcal{R}^+(u)$ (resp. $(\mathcal{R}^-(u))^{-1}$) reduces to 
$$\text{exp}_q\left((q^{-1} -
  q) x_{1,0}^+\otimes x_{1,0}^-\right)\text{ ( resp. }\text{exp}_{q^{-1}}\left((q - q^{-1})u k_1^{-1}x_{1,1}^-\otimes
  x_{1,-1}^+  k_1\right)).$$
So we get
$$\mathcal{R}^+(u).(z_j\otimes v_i') = \sum_{\lambda\geq 0} \frac{(q^{-1} - q)^\lambda (x_0^+)^\lambda z_j\otimes (x_0^-)^\lambda v_i'}{q^{\frac{\lambda (\lambda -1)}{2}}[\lambda ]_q!} = \sum_{0\leq \lambda\leq \text{Min}(j,k-i)} \frac{q^{2i\lambda + \frac{\lambda(\lambda - 1)}{2}}}{[\lambda]_q!} 
 z_{j-\lambda}\otimes v_{i+\lambda}'$$
which matches the formula for $\tau S_{W_k,L_1^+}(u^{-1})\tau$ in Example (\ref{forb}). The formula is obtained from 
$$(x_0^+)^\lambda .z_j = z_{j- \lambda }\text{ and }(x_0^-)^\lambda .v_i' = q^{2i\lambda  + \lambda (\lambda -1)}(q^{-1} - q)^{-\lambda } v_{i+\lambda }'$$ 
as $(x_0^-)^\lambda .v_i = ([i+1]_q\cdots [i+\lambda ]_q)([k-i]_q\cdots [k-(i+\lambda -1)]_q)v_{i+\lambda }$.

Note that we have $e_0 = k^{-1}x_{1,1}^-$ and $f_0 = x_{1,-1}^+k$ and
$$f_0^\lambda . v_j = q^{\lambda(k-2)}v_{j - \lambda}\text{ and }e_0^\lambda z_i' = q^{\lambda(i+1) + \frac{\lambda(\lambda +1)}{2}} (q^{-1} - q)^{-\lambda} z_{i + \lambda}'.$$
As $\mathcal{R}^\infty. (z_i'\otimes v_j) = q^{i(k - 2j)} z_i'\otimes v_j$ and $(\mathcal{R}^{\infty})^{-1} z_{i+\lambda}'\otimes v_{j - \lambda} =  q^{(i + \lambda)(2j - 2\lambda - k)}z_{i+\lambda}'\otimes v_{j - \lambda}$, we get also
$$(\mathcal{R}^{\infty})^{-1}(\mathcal{R}^-(u))^{-1}\mathcal{R}^\infty.(z_i'\otimes v_j)  
= \sum_{0\leq \lambda\leq j} \frac{u^\lambda(q - q^{-1})^\lambda (e_0)^\lambda z_i'\otimes (f_0)^\lambda v_j}{q^{\frac{-\lambda (\lambda - 1)}{2}}[\lambda ]_q!}$$
$$= \sum_{0\leq \lambda\leq j} \frac{(-u)^\lambda q^{\lambda(2j - i) - \lambda(1 +\lambda)} z_{i + \lambda}'\otimes  v_{j-\lambda}}{[\lambda ]_q!}.$$
This matches the formula for $S_{L_1^+,W_k}(u)$ in Example \ref{forb}.
\end{ex}

\section{$R$-matrices in the category $\mathcal{O}$}\label{rm}

In this section we give application of the results of the first sections in this paper to the construction of new $R$-matrices 
in the category $\mathcal{O}$ (Theorem \ref{isopre}) and to categorifications of remarkable relations (Theorem \ref{catqq}), the $QQ^*$-systems 
in the Grothendieck ring $K_0(\mathcal{O})$.

\subsection{Braidings in the category $\mathcal{O}$}\label{brr}
We have reminded above that the category $\mathcal{C}$ of finite-dimensional representations admits generic braidings.
The commutativity of the Grothendieck ring of the category $\mathcal{O}$ could indicate
the category $\mathcal{O}$ has the same property. However it is not the case.
Moreover tensor products of simple modules are not generically simple in the category $\mathcal{O}$. 

\begin{ex} For $\Glie = \hat{sl}_2$, let us study the tensor product of prefundamental representations $L_a^+\otimes L_b^-$.
This representation is never simple. Indeed its character is $(\sum_{r\geq 0}[-r\alpha])^2$ but the character of the simple representation of
the same highest $\ell$-weight is of the form $\sum_{0\leq r\leq R}[-r\alpha]$ where $R\in \mathbb{N}\cup {+\infty}$ (it can be realized a an 
evaluation representation of a simple $\mathcal{U}_q(sl_2)$-module). Hence $L_1^+(z)\otimes L_1^-$ is not generically simple.
\end{ex}

In addition, tensor product do not generically commute.

\begin{ex}\label{exnon} For $\Glie = \hat{sl}_2$, $L_1^+(z)\otimes L_1^-$ is not isomorphic to $L_1^-\otimes L_1^+(z)$. This can be proved as in \cite{BJMST}.  For completeness let us give an argument. Let $a,b\in\mathbb{C}^*$ and consider a basis $(w_i)_{i\geq 0}$ (resp. $(z_i)_{i\geq 0}$) of $L_b^-$ (resp. $L_a^+$) 
as above. Then the kernel of the action of $e_1$ on $(L_a^+\otimes L_b^-)_{-\alpha}$ (resp. on $(L_b^-\otimes L_a^+)_{-\alpha}$)
is generated by $u = z_1\otimes w_0 - z_0\otimes w_1$ (resp. $v = w_1\otimes z_0 - w_0\otimes z_1$). As $h_{1,1} = q^{-2}e_1 e_0 - e_0 e_1$, we get that $v$ is eigenvector of $h_{1,1}$ of eigenvalue $\frac{a - b}{q - q^{-1}}$ 
which is $0$ if and only if $a = b$. However
$$h_{1,1}.u = \frac{1}{q - q^{-1}} (z_0\otimes w_1(b(q^2-1-q^{-2}) + a) + z_1\otimes w_0(a(-q^2-1+q^{-2}) + b))).$$
And so $u$ is an eigenvector if and only $b(q^2-1-q^{-2}) + a) + (a(-q^2-1+q^{-2}) + b = 0$, that is $a = b$.
And in this case, the eigenvalue is $-a(q + q^{-1})$ which is non zero. $v$ generates a submodule of $L_b^-\otimes L_a^+$ of highest weight $-\alpha$. But $L_a^+\otimes L_b^-$ has no such submodule
if $a\neq b$. In the case $a = b$, it has such a submodule, but not isomorphic to the first one. We have proved that $L_a^+\otimes L_b^-$
is never isomorphic to $L_b^-\otimes L_a^+$.\end{ex}

However that there are many examples of braidings in the category $\mathcal{O}$.
For instance we have the following.

\begin{thm}\cite{FH}\label{stensor} Any tensor product of positive
  (resp. negative) prefundamental representations $L_{i,a}^+$
  (resp. $L_{i,a}^-$)
  is simple. 
\end{thm}

As a direct consequence, for each $i,j\in I$, there are isomorphisms
$$L_{i,1}^+(z) \otimes L_{j,1}^+ \simeq L_{j,1}^+\otimes L_{i,1}^+(z),$$
$$L_{i,1}^-(z) \otimes L_{j,1}^- \simeq L_{j,1}^-\otimes L_{i,1}^-(z).$$

Note also that the universal $R$-matrix $\mathcal{R}(z)$ can be generically specialized on a tensor product $V\otimes W$ of a simple finite-dimensional representation $V$ by a representation $W$ in the category $\mathcal{O}$. This leads to non zero morphisms $\mathcal{R}_{V,W}(z) : V(z)\otimes W\rightarrow W\otimes V(z)$ as in the case of finite-dimensional representations. Although $\mathcal{R}(z)$ can not be 
specialized directly on $W\otimes V(z)$, we can apply $(\mathcal{R}(z))^{-1}\circ \tau$ on $W\otimes V(z)$ to get the inverse. So we get an isomorphism. We can also consider as above the associated specialization $\mathcal{R}_{V,W}$ which is not invertible in general. 

\begin{ex}\label{catun} For $\Glie = \hat{sl}_2$ we may consider the case of $V= L(Y_{1,aq})$ fundamental representation of dimension $2$ and $W = L_{1,a}^-$ prefundamental representation.
Then $\mathcal{R}_{V,W}$ has simple image and kernel isomorphic respectively to $L_{1,aq^{-2}}^-\otimes [\omega_1]$ and $L_{1,aq^2}^-\otimes [-\omega_1]$. 
This leads to an exact sequence
$$0\rightarrow L_{1,aq^2}^-\otimes [-\omega_1] \rightarrow L(Y_{1,aq})\otimes L_{1,a}^-\rightarrow L_{1,aq^{-2}}^-\otimes [\omega_1]\rightarrow 0$$
which is a categorification of the Baxter's $QT$-relation in $K_0(\mathcal{O})$ (see \cite[Remark 4.10]{FH}) : 
$$[L(Y_{1,aq})][L_{1,a}^-]  = [L_{1,aq^{-2}}^-] [\omega_1]  + [L_{1,aq^2}^-] [-\omega_1].$$
\end{ex}

For general types, there are various generalization of the Baxter's $QT$-relation, such as the generalized Baxter's relations \cite{Fre, FH} 
or the $QQ^*$-systems considered in \cite[Section 6.1.3, Example 7.8]{HL} 
from the point of view of cluster algebras (they are obtained as cluster mutation relations, see \cite{l} for a general
point of view). 
They involve the simple representation $W = L_{i,q^r}^-$ and the simple representation $L_{i,a}^* = L(Y_{i,aq_i} \prod_{j, C_{j,i}< 0} \Psib_{j,aq_j^{C_{j,i}}})$ which is not finite-dimensional 
(except in the $sl_2$-case). The relation reads  
\begin{equation}\label{QQ}[L_{i,a}^*][L_{i,a}^-] = [\omega_i] \prod_{j, C_{j,i}\neq 0} \left[L_{j,aq_j^{C_{j,i}}}^-\right]   + [\omega_i-\alpha_i]\prod_{j, B_{j,i}\neq 0} \left[L_{j,aq_j^{- C_{j,i}}}^-\right].\end{equation}
Note that the $QQ^*$-systems are important not only from the cluster algebras point of view, 
but they also lead to the Bethe Ansatz equations \cite{FJM}.

This a motivation to construct $R$-matrices in a more general situation (see Section \ref{real} below). The relevant framework seems to be the monoidal subcategory 
$\mathcal{O}^-$ of the category $\mathcal{O}$ defined in \cite{HL}.

\begin{defi}\label{omoins}\cite{HL} The category $\mathcal{O}^-$ is the full subcategory of representations in the category $\mathcal{O}$ whose image 
in $K_0(\mathcal{O})$ are in the subring generated by finite-dimensional representations and the prefundamental representations $L_{i,a}^-$, $i\in I$, $a\in\mathbb{C}^*$. 
\end{defi}

The generalized Baxter's relations as well as the $QQ^*$-systems 
hold in the Grothendieck ring $K_0(\mathcal{O}^-)$. Moreover this ring has nice properties in the context of cluster algebras (see \cite{HL, bit}). 

\subsection{$R$-matrices by stable maps}

We would like to know how to construct braidings when the universal $R$-matrix can not be directly specialized.
To attack this problem, mimicking the approach of Maulik-Okounkov, the algebraic stable maps give a natural path.

Let $V$, $W$ be simple representations in the category $\mathcal{O}$. Then the space $V\otimes W$ has a structure of 
$\U_q(\Hlie^+)$-module from the Hopf-algebra structure of $\U_q(\bo)$. But it has also another 
$\U_q(\Hlie^+)$-module structure obtained from the Drinfeld coproduct\footnote{In simply-laced case, a geometric approach to the Drinfeld coproduct is 
proposed  in \cite{VV}.} 
$$\Delta_d : \U_q(\Hlie^+)\rightarrow  \U_q(\Hlie^+)\otimes  \U_q(\Hlie^+)$$ 
defined by
\begin{equation}\label{deltad}\Delta_d(h_{i,r}) = h_{i,r}\otimes 1 + 1 \otimes h_{i,r}\text{ , }\Delta_d(k_i) = k_i\otimes k_i.\end{equation}
Let us denote by $V\otimes_d W$ the corresponding $\U_q(\Hlie^+)$-module. Similarly, we can define a representation $V(u)\otimes_d W$.

\begin{rem} Recall the filtration of $V\otimes W$ by $\mathcal{U}_q(\Hlie)^+$-submodules associated to the partial ordering $\preceq$ 
in the proof of Proposition \ref{isom}. 
Then the $\mathcal{U}_q(\Hlie)^+$-module $V\otimes_d W$ is isomorphic to the graded module associated to this filtration.
\end{rem}

Let $\alpha(u)$ be an automorphism of the $\U_q(\Hlie)^+$-module $V(u)\otimes_d W$ and consider the composition
$$\xymatrix{ V(u)\otimes W \ar[d]_{S_{V,W}^{-1}(u)}\ar[r]^{I_{V,W}^\alpha(u)} &  W\otimes V(u)
\\V(u)\otimes_d W\ar[r]^{\tau\circ \alpha(u)} & W\otimes_d V(u)\ar[u]_{S_{W,V}(u^{-1})}},$$
where $\tau$ is the twist. We get a linear isomorphism
$$I_{V,W}^\alpha(u) : V(u)\otimes W \rightarrow W\otimes V(u).$$ 
These are candidates for $R$-matrices in the category $\mathcal{O}$, but we have to make good choices for $\alpha(u)$.

To illustrate this, consider $V$, $W$ simple finite-dimensional representations of the full quantum affine algebra $\U_q(\Glie)$. 
Recall that by Corollary \ref{ratfact}, the action of $\mathcal{R}^0(u)$ on $V(u)\otimes W$ is rational up to a scalar factor which is the eigenvalue of 
the tensor product of highest weight vectors. 
We will work with the rational part that we denote by $\overline{\mathcal{R}}^0(u)$ (it depends on $V$ and $W$, there is a slight abuse of notation). We get a diagram
$$\xymatrix{ V(u)\otimes W \ar[d]_{(\mathcal{R}^\infty)^{-1}\mathcal{R}^-(u)\mathcal{R}^\infty}\ar[r]^{I_{V,W}^\alpha(u)} &  W\otimes V(u)
\\V(u)\otimes W\ar[r]^{\tau\circ \overline{\mathcal{R}}^0(u)\mathcal{R}^\infty} & W\otimes V(u)\ar[u]_{\tau\circ\mathcal{R}^+(u)\circ\tau}}.$$
The composition is equal up to a scalar to
$$S_{W,V}(u^{-1})  \tau \overline{\mathcal{R}}^0_{V,W}(u)\mathcal{R}^\infty   S_{V,W}^{-1}(u) 
=\tau\mathcal{R}^+(u)\tau\tau \overline{\mathcal{R}}^0(u)\mathcal{R}^-(u)\mathcal{R}^\infty.$$
It coincides with the action of $\tau\mathcal{R}(u)$ up to a scalar and so it is an isomorphism of $\U_q(\Glie)$-modules. 

\begin{ex} We continue Example \ref{exfam}. In the same basis, the matrix of $\mathcal{R}^0(u)$ is
$$\begin{pmatrix}\text{exp}\left(\sum_{m > 0}\frac{u^m (q^m - q^{-m})q^{-2m}}{m(q^m + q^{-m})}\right) & 0\\ 0 & \text{exp}\left(\sum_{m > 0}\frac{u^m (q^m - q^{-m})q^{2m}}{m(q^m + q^{-m})}\right)\end{pmatrix}$$ 
$$= \text{exp}\left(\sum_{m > 0}\frac{u^m (q^m - q^{-m})q^{2m}}{m(q^m + q^{-m})}  \right)\begin{pmatrix}\frac{(1 - uq^2)(1-uq^{-2})}{(1 - u)^2} & 0 \\ 0 & 1\end{pmatrix}.$$
So for the $R$-matrix $\mathcal{R}(u) = \mathcal{R}^+(u)\mathcal{R}^0(u)\mathcal{R}^-(u)\mathcal{R}^\infty$ 
we recover the well-known matrix up to a scalar factor : 
$$\begin{pmatrix}\frac{q^{-1}(u-1)}{u - q^{-2}} &\frac{1 - q^{-2}}{u - q^{-2}}\\\frac{u(1 - q^{-2})}{u - q^{-2}}&\frac{q^{-1}(u-1)}{u - q^{-2}} \end{pmatrix}.$$
\end{ex}

\subsection{Morphism for the Cartan-Drinfeld subalgebra} 

Consider $V$ and $W$ simple modules in the category $\mathcal{O}$. 
Recall the category $\mathcal{O}^-$ in Definition \ref{omoins}. 

The statement of Proposition \ref{caract} is also true for simple representations in the category $\mathcal{O}^-$. 
Indeed, the proof relies on the property (\ref{infol}) of $\ell$-weights of simple finite-dimensional representations which is 
also satisfied for simple representations in the category $\mathcal{O}^-$ \cite[Section 7.2]{HL}. Consequently, the
statement of Proposition \ref{uni} (uniqueness of $\ell$-weight vectors) is also satisfied.


\begin{prop}\label{hiso} Suppose that one of the simple representations $V$ or $W$ is finite-dimensional, or more generally in the category $\mathcal{O}^-$. Then 
$$S_{V,W}(u): V(u)\otimes_d W\rightarrow V(u)\otimes W$$ 
is an isomorphism of $\U_q(\Hlie^+)$-modules.\end{prop}

\begin{proof} Let $v\otimes w\in V_{\Psib}\otimes W_{\Psib'}$. Then 
$$v\otimes w \in V(u)\otimes_d W$$ 
and 
$$S_{V,W}(u)(v\otimes w)\in V(u)\otimes W$$ 
are $\ell$-weight vectors of $\ell$-weight $\Psib(u)\Psib'$. 
Then $S_{V,W}(u)$ defines a linear isomorphism 
$$S_{V,W}(u) : (V_{\Psib}(u) \otimes W_{\Psib'}) = (V(u)\otimes_d W)_{\Psib(u)\Psib'}\rightarrow (V(u)\otimes W)_{\Psib(u)\Psib'}$$
between the corresponding $\ell$-weight spaces.
Now it follows from the coproduct formula (\ref{h}) that for any $i\in I$, $m > 0$, we have
$$(h_{i,m}.S_{V,W}(u) - S_{V,W}(u).h_{i,m})(v\otimes w) \in (v\otimes w)_\prec \cap (V(u)\otimes W)_{\Psib(u)\Psib'}.$$
But from the hypothesis, this intersection is zero  as in the proof of Proposition \ref{uni}. Hence the result.
\end{proof}

It implies that $S_{V,W}^{norm}$ is a non-zero morphism of $\U_q(\Hlie^+)$-modules.

As another consequence, $I_{V,W}^\alpha(u)$ is an isomorphism of $\U_q(\Hlie)^+$-modules if $V$ or $W$ is in the category $\mathcal{O}^-$, for any $\alpha$ 
as in the previous section. In particular, for $\alpha = \text{Id}$ and $I_{V,W}(u) = I_{V,W}^{\text{Id}}(u)$, we get the following.

\begin{thm}\label{isomh} Suppose one of the simple representations $V$ or $W$ is in the category $\mathcal{O}^-$. We get an isomorphism of $\U_q(\Hlie^+)$-modules 
\begin{equation}\label{diag}\xymatrix{ V(u)\otimes W \ar[d]_{S_{V,W}^{-1}(u)}\ar[r]^{I_{V,W}(u)} &  W\otimes V(u)
\\V(u)\otimes_d W\ar[r]^\tau & W\otimes_d V(u)\ar[u]_{S_{W,V}(u^{-1})}}.\end{equation}
\end{thm}

\begin{rem} As discussed above, the map $I_{V,W}(u)$ may have poles.
\end{rem}

\begin{ex} Although they are not isomorphic as $\U_q(\bo)$-modules (see Example \ref{exnon}), $L_{i,1}(z)^+\otimes L_{j,1}^-$ and 
$L_{j,1}^-\otimes L_{i,1}(z)^+$ are isomorphic as $\U_q(\Hlie)^+$-modules.
\end{ex}

\subsection{Braidings in the category $\mathcal{O}^-$}

The following result is one of the main applications of the constructions in this paper.

\begin{thm}\label{isopre} For $V$ and $W$ simple representations in the category $\mathcal{O}^-$, there is $\alpha(u)$ 
automorphism of the $\U_q(\Hlie)^+$-module $V(u)\otimes_d W$ so that 
$$I_{V,W}^\alpha(u) : V(u)\otimes W \rightarrow W\otimes V(u)$$ 
is an isomorphism of $\U_q(\bo)$-modules.
\end{thm}

\begin{proof} Let us recall each simple module $L(\Psib)$ in the category $\mathcal{O}^-$, there is a sequence of finite-dimensional modules
constructed in \cite[Section 7.2]{HL} whose $q$-characters converge to $\chi_q(L(\Psib))$, up to a normalization, see \cite[Theorem 7.1]{HL} 
(in the case when $L(\Psib)$ is a prefundamental representation, it is a sequence of Kirillov-Reshetikhin modules considered in 
\cite[Section 4.1]{HJ}). Consider the sequences $V_k = L(N_k)$, $W_k = L(M_k)$ associated respectively to $V$, $W$. 
Then for $l\geq k$ we get as in \cite[Section 4.2]{HJ} an injective linear morphism
$$F_{l,k} :  W_k \rightarrow  W_l.$$
It is obtained as the composition of the surjective morphism 
$$W_k\otimes L(M_lM_k^{-1})\rightarrow W_l$$ 
by the embedding 
$$W_k\rightarrow W_k\otimes v_{l-k}$$ 
where $v_{l-k}$ is a fixed highest weight vector of $L(M_l M_k^{-1})$.
As established in \cite{HJ}, $F_{l,k}$ is compatible with the $x_{i,m}^+$ and satisfies for any $j\in I$ :
$$\phi_j^+(z)\circ F_{l,k} = (M_l M_k^{-1})(\phi_j^+(z))\times (F_{l,k} \circ \phi_j^+(z)),$$
where we remind that the scalar $(M_l M_k^{-1})(\phi_j^+(z))\in\mathbb{C}((z))$ is the eigenvalue of $\phi_j^+(z)$ on an $\ell$-weight vector of corresponding monomial $M_l M_k^{-1}$. 

In the same way, we have injective linear morphisms $F_{l,k}' : V_k\rightarrow V_l$ for $l\geq k$, with the same properties.

Now, for $l\geq k$ again, we have an injective linear morphism
$$G_{l,k} : V_k\otimes W_k \rightarrow V_l \otimes W_l$$
obtained as a composition 
$$L(N_k)\otimes L(M_k) \rightarrow L(N_k)\otimes L(M_k)\otimes L(N_l/N_k)\otimes L(M_l/M_k)$$
$$\rightarrow L(N_k)\otimes L(N_l/N_k)\otimes L(M_k)\otimes L(M_l/M_k)
\rightarrow L(N_l)\otimes L(M_l).$$
The first arrow is constructed as above by using a highest weight vector of $L(N_l/N_k)\otimes L(M_l/M_k)$. 
The last two arrows are morphisms of representations (which exist as $L(M_k)\otimes L(N_l/N_k)$, 
$L(N_k)\otimes L(N_l/N_k)$ and $L(M_k)\otimes L(M_l/M_k)$ are cyclic).
Then $G_{l,k}$ has properties analogous to $F_{l,k}$. In the same way, we have
$$G_{l,k}' : W_k\otimes V_k\rightarrow W_l\otimes V_l,$$
and corresponding deformations $G_{l,k}(u)$, $G_{l,k}'(u)$, $F_{l,k}(u)$.
These maps commute with $\phi_j^+(z)$ up
to a scalar multiple, this is enough to characterize the algebraic stable maps 
$$S_{W_k,V_k}(u) : W_k\otimes_d V_k(u)\rightarrow W_k\otimes V_k(u)$$
which are constructed from the action of the Cartan-Drinfeld subalgebra. 
Then for $l\geq k$, we have
$$G_{l,k}(u)\circ S_{W_k,V_k}(u) = S_{W_l,V_l}(u)\circ (F_{l,k}\otimes F'_{l,k}(u)).$$
Precisely, both maps send a tensor product of $\ell$-weight vectors to the projection on the 
corresponding $\ell$-weight space in $W_l\otimes V_l(u)$.
This means that $S_{W_k,V_k}(u)$ is stationary when $k\rightarrow +\infty$ 
(this can be already observed in examples for the $sl_2$-case in Section \ref{exsldeux}).

Now, we have an isomorphism of finite-dimensional representations
$$\mathcal{R}_k(u) : V_k(u)\otimes W_k\rightarrow W_k\otimes V_k(u).$$
We establish by induction on the height of a weight space that for $l\geq k$ we have
$$\mathcal{R}_l(u)\circ G_{l,k}(u) = G_{l,k}'(u)\circ \mathcal{R}_k(u).$$
Indeed, we remind that there are no primitive vectors in $W_k\otimes V_k(u)$ which are not highest weight vectors
and we observe the following, for $i\in I$, $m\in \mathbb{Z}$ :
$$x_{i,m}^+\mathcal{R}_l(u) G_{l,k}(u) = \mathcal{R}_l(u)G_{l,k}(u)x_{i,m}^+ = G_{l,k}'(u)\mathcal{R}_k(u)x_{i,m}^+ = x_{i,m}^+G_{l,k}'(u)\mathcal{R}_k(u).$$
Hence we obtain stationary operators
$$ S_{W_k,V_k}^{-1}(u^{-1})\circ \mathcal{R}_k(u)\circ S_{V_k,W_k}(u)  : V_k(u)\otimes_d W_k\rightarrow W_k\otimes_d V_k(u)$$
with a well-defined limit 
$$\alpha(u) : V(u)\otimes_d W\rightarrow W\otimes_d V(u)$$
(which can be computed explicitly from the abelian part $\mathcal{R}^0$ of the universal $R$-matrix).
By construction, the corresponding composition $I_{V,W}^\alpha(u)$ is a morphism of representations. 
This implies the result.


 



\end{proof}

\begin{rem} One can make explicit the fact that $I_{V,W}^\alpha(u)$ is a morphism of representations. 
As the Borel algebra is generated by its intersection with the asymptotical algebra $\tilde{\mathcal{U}}_q(\mathfrak{g})$ of \cite{HJ} 
and by the Cartan subalgebra, it suffices to consider $g$ in this intersection. Then for $v\in V\otimes W$ we have  
$$g.v = Lim_{k\rightarrow \infty} G_{\infty,k}g.G_{\infty,k}^{-1}v.$$ 
Besides, one has $I_{V,W}^\alpha G_{\infty,k} = G_{\infty,k}' R_k$ on each weight space for $k$ large enough. Hence
$$g.(I_{V,W}^\alpha.v) = \text{Lim}_{k\rightarrow \infty} G'_{\infty,k} g (G'_{\infty,k})^{-1}I_{V,W}^\alpha v
= \text{Lim}_{k\rightarrow \infty} G'_{\infty,k}g R_k(u)G_{\infty,k}^{-1} v$$
$$= \text{Lim}_{k\rightarrow \infty} G'_{\infty,k}R_k(u)gG_{\infty,k}^{-1} v
= \text{Lim}_{k\rightarrow \infty} I_{V,W}^\alpha (G_{\infty,k} g G_{\infty,k}^{-1} v)
= I_{V,W}^\alpha(g.v).$$
\end{rem}

\begin{rem}\label{remce} There are counter-examples when the representations are not in the category $\mathcal{O}^-$ : 
the fact that $\mathcal{R}^0(u)$ converges does not imply that we get a morphism. For example in the $sl_2$-case 
consider the limit when $k\rightarrow +\infty$ of algebraic stable maps on $L_1^+(u)\otimes W_k$ where $W_k$ converges to $L_1^-(u)$.  We can use 
$$\mathcal{R}^0(u) = \text{exp}\left( - (q - q^{-1}) \sum_{m > 0} u^m \frac{m}{[m]_q(q^m + q^{-m})}h_{1,m}\otimes h_{1,-m}\right).$$
Each operator $h_{1,m}$ has a scalar action $\frac{\text{Id}}{m(q^{-1}  - q)}$ on $L_1^+$. Hence we get the operator
$$\text{Id}\otimes \text{exp}\left( \sum_{m > 0} \frac{u^m}{[m]_q(q^m + q^{-m})} h_{1,-m}\right).$$
The space $W_k$ has a basis $(v_j)_{0\leq j\leq m}$ of eigenvectors of $\phi^-(z)$ with eigenvalue
$$q^{2j - k}\frac{(1 - q^{2k}z^{-1})(1-q^{-2}z^{-1})}{(1 - q^{2j - 2}z^{-1})(1 - q^{2j}z^{-1})}.$$
and so the eigenvalue of $(q^{-1} - q) h_{1,-m}$ on $v_j$ is $ q^{2(j - 1)m} + q^{2jm} - q^{2km} - q^{-2m}$. 
Then $L_1^+(u)\otimes v_j$ is an eigenspace of $\mathcal{R}^0(u)$ with  eigenvalue
$$\text{exp}\left(  \sum_{m > 0} \frac{u^m }{q^{2m} + q^{-2m}} (-q^{2(j - 1)m} - q^{2jm} + q^{2km} + q^{-2m})\right).$$
So if we set
$$\alpha(k) = \text{exp}\left(  \sum_{m > 0} \frac{u^m }{q^{2m} - q^{-2m}} q^{2km} \right),$$
the action of the operator $\alpha(k) \mathcal{R}_0(u)$ does not depend on $k$. This gives a well-defined
automorphism of the $\U_q(\Hlie)^+$-module $L_1^+(u)\otimes_d L_1^-$.
\end{rem}

\begin{rem} It should also be possible to derive from \cite[Theorem 7.6]{HL} that a tensor product of simple representations in the category $\mathcal{O}^-$
is generically simple. Our result gives in addition a construction of corresponding braidings as well as a factorization of these braidings using algebraic stable maps.
\end{rem}

We will denote the $R$-matrix we have constructed by $I_{V,W}^\alpha(u) = \mathcal{R}_{V,W}(u)$.  
As above, we can consider the first term in the development in $u - 1$ (see also Section \ref{normsp}).
We get a non-zero morphism in the category $\mathcal{O}$ : 
$$\mathcal{R}_{V,W} : V\otimes W\rightarrow W\otimes V$$
which is not invertible in general.

\subsection{Example : braidings and $QQ^*$-systems}\label{real}

 We have seen in Example \ref{catun} that in the $sl_2$-case the Baxter's QT-relation can be categorified using a normalized
$R$-matrix. As an application of the above result, we obtain also categorified versions of the $QQ^*$-systems for general types (see section \ref{brr} 
and Equation (\ref{QQ})).

\begin{thm}\label{catqq} The specialized $R$-matrix
$$\mathcal{R}_{L_{i,a}^*,L_{i,a}^-} : L_{i,a}^*(u)\otimes L_{i,a}^-\rightarrow L_{i,a}^-\otimes L_{i,a}^*(u)$$
is non invertible and gives a non-splitted exact sequence
$$0\rightarrow  [\omega_i-\alpha_i]\bigotimes_{j, B_{j,i}\neq 0} L_{j,aq_j^{- C_{j,i}}}^- \rightarrow L_{i,a}^*\otimes L_{i,a}^- \rightarrow 
[\omega_i] \bigotimes_{j, C_{j,i}\neq 0} L_{j,aq_j^{C_{j,i}}}^- \rightarrow 0$$
which categorifies the $QQ^*$-system (\ref{QQ}).
\end{thm}



\begin{proof}
From the $QQ^*$-system and Theorem \ref{stensor}, the tensor product $L_{i,a}^*\otimes L_{i,a}^-$ 
is of length $2$. Hence the image of the specialized braiding 
$\mathcal{R}_{L_{i,a}^*,L_{i,a}^-}$ is simple or isomorphic to $L_{i,a}\otimes L_{i,a}^*$. 

Let $\Psib$ be the highest $\ell$-weight of $L_{i,a}^*\otimes L_{i,a}^-$. We will also discuss the following $\ell$-weights : 
$$\Psib' = \Psib A_{i,a}^{-1}\text{ , }\Psib'' = \Psib A_{i,a q_i^2}^{-1}\text{ , }\Psib''' = \Psib' A_{i,a}^{-1},$$
where the $A_{i,a}$ are defined as in Section \ref{qchar}. By the analysis in \cite[Section 6.1.3, 7.2]{HL}, these are $\ell$-weights of $L_{i,a}^*\otimes L_{i,a}^-$ of corresponding $\ell$-weight spaces of dimension $1$. 
The two simple constituents of the tensor product are $L(\Psib)$ and $L(\Psib')$. The $\ell$-weights $\Psib''$ and $\Psib'''$ are $\ell$-weights of $L(\Psib)$ only.

Consider the representation $L_{i,a}^-\otimes L_{i,a}^*$. Let $w_{i,a}$ be an highest weight vector of $L_{i,a}^*$ and 
$v_{i,a}'$ a weight vector of $L_{i,a}^-$ of weight $-\alpha_i$. From Theorem \ref{prodlweight}, we get
$$S^{norm}_{L_{i,a}^-, L_{i,a}^*} (v_{i,a}'\otimes w_{i,a}) = v_{i,a}'\otimes w_{i,a}$$
which generates the $\ell$-weight space associated to $\Psib'$. But 
$$x_{i,0}^+.(v_{i,a}'\otimes w_{i,a}) = (x_{i,0}^+.v_{i,a}')\otimes w_{i,a} \neq 0.$$
Hence, such a vector is not of highest $\ell$-weight. This implies that $L_{i,a}^-\otimes L_{i,a}^*$ is cocyclic, that is the submodule generated by a tensor product
of highest weight vectors is simple.

Consider now the representation $L_{i,a}^*\otimes L_{i,a}^-$. Let $v_{i,a}$ be a highest weight vector of $L_{i,a}^-$ and 
$w_{i,a}'$ a weight vector of $L_{i,a}^*$ of weight $-\alpha_i$. As above, 
$$S^{norm}_{L_{i,a}^*, L_{i,a}^-} (w_{i,a}'\otimes v_{i,a}) = w_{i,a}'\otimes v_{i,a},$$
$$S^{norm}_{L_{i,a}^*, L_{i,a}^-} (v_{i,a}'\otimes v_{i,a}') = w_{i,a}'\otimes v_{i,a}',$$
which generate the $\ell$-weight spaces associated respectively to $\Psib''$ and $\Psib'''$ 
(note however that it would be more complicated for the $\ell$-weight associated to $\Psib''' A_{i,aq_i^{-2}}^{-2}$).
But $x_{i,0}^+ (w_{i,a}'\otimes v_{i,a}') \notin \mathbb{C} . w_{i,a}'\otimes v_{i,a}$. Hence $L_{i,a}^*\otimes L_{i,a}^-$ is not cocyclic, 
but cyclic, that is generated by a tensor product of highest weight vectors. 

We can conclude : the two representations are not isomorphic, the image of $\mathcal{R}_{L_{i,a}^*,L_{i,a}^-}$ is simple isomorphic to $L(\Psib)$.
\end{proof}

\section{Further directions}\label{fd}

In this section we discuss various possible further developments of the results in this paper.

\medskip

{\bf Polynomiality.} A polynomiality property of the action of Cartan-Drinfeld elements was established in \cite[Theorem 5.17]{FH} : the action of a certain
family of Cartan-Drinfeld current $T_i(z)$, which characterize the action of the Cartan-Drinfeld algebra $\U_q(\Hlie)^+$, act polynomially
on any tensor product $W$ of simple-finite dimensional modules. The relation to the polynomiality of the algebraic stable maps $S_{L_{i,1}^+,W}(u)$, $S_{W,L_{i,1}^+}(u)$ observed in the $sl_2$-case 
(section \ref{exsldeux}) has to be understood.

\medskip

{\bf Baxter algebra and geometry.} One of the main application of the theory of Maulik-Okounkov is the relation to the action of the Baxter subalgebra and to its eigenvectors \cite{mo}. 
The Baxter subalgebra is generated by coefficients of transfer-matrices and can be seen as a deformation of the Cartan-Drinfeld subalgebra $\U_q(\Hlie)^+$. A natural question is to study in this context the relation between $\ell$-weight vectors and eigenvectors of the Baxter algebra.
More generally, a geometric framework for the result of the present paper has to be developed, as for example the results 
obtained in \cite{psz} for the prefundamental representations in type $A$. 
We hope our results give additional practical tools to handle the corresponding geometric structures. 
The case of non symmetric cases is open as well.

\medskip

{\bf Fusion product.} A fusion product $*$ was defined in \cite{htg} for finite-dimensional modules of highest $\ell$-weight from a specialization of the Drinfeld coproduct. It would be interesting 
to understand how algebraic stable maps behave relatively to this
structure, for instance to determine if $S_{V,W} : V * W\rightarrow V\otimes W$ defines a morphism. 

\medskip

{\bf Generalized Schur-Weyl dualities.} Kang-Kashiwara-Kim defined in \cite{kkk} generalized Schur-Weyl dualities 
as functors from categories of representations of quiver Hecke-algebras (Khovanov-Lauda-Rouquier algebras) to 
categories of finite-dimensional representations of quantum affine algebras, generalizing previous results of Chari-Pressley \cite{Chasw} obtained in type $A$. This leads to very interesting equivalences of categories. 
The construction of the generalized Schur-Weyl functors is based on certain bimodules obtained from the braidings 
in the category $\mathcal{C}$ of finite-dimensional representations. The braidings constructed in this
paper for the category $\mathcal{O}^-$ (Theorem \ref{isopre}) should lead to an extension of the 
construction of \cite{kkk} and to possible equivalences between subcategories of the category $\mathcal{C}$ and of the category $\mathcal{O}^-$, explaining seemly analogous structures.

\medskip

{\bf Tensor products and basis of $\ell$-weight vectors.} Using the framework of the present paper, one can define algebraic stable maps $S_{V_1,V_2,\cdots, V_N}$ on tensor products of more than $2$ factors $V_1\otimes V_2\otimes \cdots \otimes V_N$  as well as the corresponding deformations 
$$S_{V_1,V_2,\cdots, V_N}(u_1,\cdots, u_N).$$ 
For $i < j$, we have the algebraic stable map $S_{V_i,V_j}(u_i,u_j)$. After tensoring with identity maps, it gives an operator 
$$S_{V_1,V_2,\cdots, V_N}^{(i,j)}(u_i,u_j).$$ 
We conjecture that the composition of such operators is equal to $S_{V_1,V_2,\cdots, V_N}(u_1,\cdots, u_N)$. 
\\Besides, for a family of simple modules $V_1,\cdots, V_N$ endowed with a basis of $\ell$-weight vectors, the algebraic stable map $S_{V_1,V_2,\cdots, V_N}$ gives such a basis of $\ell$-weight vectors of the tensor product $V_1\otimes \cdots \otimes V_N$. For example, one may consider a family of thin (that is with one dimensional $\ell$-weight subspaces) fundamental modules. We get a 
natural basis of $\ell$-weight vectors in the corresponding standard module, that is the tensor product of the fundamental modules. In types $A, B, C, G_2$, 
all fundamental modules are thin, and so we get a basis of all standard modules. More generally, an arbitrary simple module is a subquotient of such a standard module (except in type $E_8$ by \cite[Proposition 7.3]{FH}). We intend to study if such bases of standard modules descend to simple modules and how such bases behave relatively to tensor products.

\end{document}